\documentclass[12pt]{amsart}
\usepackage{mathrsfs, amscd,amssymb,amsthm,amsmath,csquotes,graphicx,array,wasysym, textcomp,comment,epigraph,rotating,fancyvrb,pdflscape,upgreek}
\usepackage{longtable}
\usepackage{adjustbox}
\usepackage{tabularx,hhline,multirow}

\usepackage[margin=2cm]{geometry}
\usepackage[matrix,arrow,curve]{xy}

\usepackage{hyperref}
\usepackage{graphicx}
\newcommand{\OOO}{\mathscr{O}}
\newcommand{\Aff}{\mathbb{A}}
\newcommand{\CC}{\mathbb{C}}
\newcommand{\ZZ}{\mathbb{Z}}
\newcommand{\PP}{\mathbb{P}}
\newcommand{\QQ}{\mathbb{Q}}
\newcommand{\NN}{\mathbb{N}}
\newcommand{\FF}{\mathbb{F}}
\newcommand{\UU}{\mathbb{U}}
\newcommand{\BB}{\mathbb{B}}
\newcommand{\Ga}{\mathbb{G}_{\mathrm{a}}}
\newcommand{\Gm}{\mathbb{G}_{\mathrm{m}}}
\newcommand{\type}[1]{\mathrm{#1}}
\newcommand{\mumu}{\boldsymbol{\mu}}
\renewcommand{\rho}{\uprho}
\newcommand{\simQ}{\sim_{\QQ}}

\newcommand{\NE}{\overline{\operatorname{NE}}}
\newcommand{\Aut}{\operatorname{Aut}}
\newcommand{\mult}{\operatorname{mult}}
\newcommand{\Pic}{\operatorname{Pic}}
\newcommand{\Cl}{\operatorname{Cl}}
\newcommand{\rk}{\operatorname{rk}}
\newcommand{\GL}{\operatorname{GL}}

\newcommand{\PGL}{\operatorname{PGL}}

\newcommand{\Sing}{\operatorname{Sing}}
\newcommand{\Type}{\operatorname{Type}}
\newcommand{\numl}{\operatorname{\#}}

\newcommand{\tors}{\mathrm{tors}}
\newcommand{\red}{\mathrm{red}}

\newcommand{\ind}{\tau}

\newcommand{\xref}[1]{\textup{\ref{#1}}}

\newtheorem*{maintheorem}{Main Theorem}
\newtheorem{theorem}[equation]{Theorem}

\newtheorem{proposition}[equation]{Proposition}
\newtheorem{lemma}[equation]{Lemma}
\newtheorem{corollary}[equation]{Corollary}

\theoremstyle{definition}
\newtheorem{example}[equation]{Example}
\newtheorem{definition}[equation]{Definition}

\newtheorem{remark}[equation]{Remark}

\makeatletter\@addtoreset{equation}{section} \makeatother

\newcounter{No}
\renewcommand{\theNo}{{{\rm\arabic{No}$^o$}}}
\def\no{\refstepcounter{No}{\theNo}}

\author{Ivan Cheltsov}
\address{Ivan Cheltsov, University of Edinburgh, Edinburgh, UK}
\email{I.Cheltsov@ed.ac.uk}

\author{Yuri Prokhorov}
\address{Yuri Prokhorov, Steklov Mathematical Institute, Moscow, Russia}
\email{prokhoro@mi-ras.ru}

\title{Del Pezzo surfaces with infinite automorphism groups}

\begin{document}

\begin{abstract}
We classify del Pezzo surfaces with Du Val singularities that have infinite automorphism groups,
and describe the~connected components of their automorphisms groups.
\end{abstract}

\maketitle
\setcounter{tocdepth}1

Throughout this paper, we always assume that all varieties are projective and defined over an~algebraically closed field $\Bbbk$ of characteristic $0$.

\section{Introduction}
\label{section:intro}

Automorphism groups of smooth del Pezzo surfaces are well-studied (see, for example,~\cite{Dolgachev,DolgachevIskovskikh}).
In~particular, if $X$ is a~smooth del Pezzo surface, then $\Aut(X)$ is infinite if and only if $X$ is toric.
Moreover, if $X$ is a~smooth toric del Pezzo surface, then $\Aut^0(X)$ can be described as follows:

\smallskip

\begin{center}\renewcommand{\arraystretch}{1.2}
\begin{tabularx}{0.80\textwidth}{|c|c|X|c|X|}
\hline
$K_X^2$&$\Aut^0(X)$&\multicolumn2{c|}{equation \& total space}
\\
\hhline{|=|=|=|=|=|}
\hline
$6$&$\Gm^2$&\centering$u_0v_0w_0=u_1v_1w_1$&$\PP^1\times\PP^1\times\PP^1$
\\
\hline
$7$&$\BB_2\times\BB_2$&&
\\
\hline
$8$&$\Ga^2\rtimes\GL_2(\Bbbk)$&\centering $u_0v_0=u_1v_1$&$\PP^2\times\PP^1$
\\
\hline
$8$&$\PGL_2(\Bbbk)\times\PGL_2(\Bbbk)$&\centering---&$\PP^1\times\PP^1$
\\
\hline
$9$&$\PGL_3(\Bbbk)$&\centering---&$\PP^2$
\\
\hline
\end{tabularx}
\end{center}

\smallskip

\noindent where $\Ga$ is a~one-dimensional unipotent additive group,
$\Gm$ is a~one-dimensional algebraic torus, and $\BB_2$ is the~Borel subgroup of $\PGL_2(\Bbbk)$.
In this paper, we prove similar result for del Pezzo surfaces with at worst Du Val singularities.
For short, we call such surfaces \textit{Du Val del Pezzo surfaces}.
Our main result is the~following.

\begin{maintheorem}
\label{theorem:main}
Let $X$ be a~Du Val del Pezzo surface.
Then the~group $\Aut(X)$ is infinite if and only if $X$ is described in Big Table in Section~\xref{section:tables}.
\end{maintheorem}
Everywhere below the~number $n^0$ refers to the~corresponding surface in Big Table in Section~\xref{section:tables}.
As a consequence of our classification we have the~following

\begin{corollary}
\label{corollary:cscK-metric}
Let $X$ be a~Du Val del Pezzo surface.
Then the~group $\Aut(X)$ is not reductive if and only if $X$ is one of the~$23$ surfaces
\ref{d=2:A7}, \ref{d=3:E6}, \ref{d=3:A5-A1}, \ref{d=3:A5}, \ref{d=4:D5}, \ref{d=4:A3-2A1}, \ref{d=4:D4}, \ref{d=4:A4}, \ref{d=4:A3-A1},
\ref{d=4:A3-4lines}, \ref{d=5:A4}, \ref{d=5:A3}, \ref{d=5:A2-A1}, \ref{d=5:A2}, \ref{d=6:A2-A1}, \ref{d=6:A2}, \ref{d=6:2A1}, \ref{d=6:A1-3l}, \ref{d=6:A1-4l}, \ref{d=7}, \ref{d=7:smooth}, \ref{d=8}, \ref{d=8:F1}.
\end{corollary}

Thus, the~surfaces listed in this corollary are not~$K$-polystable \cite{AlperBlumHalpernLeistnerXu,Matsushima}, which is known (see \cite{OdakaSpottiSun}).

\begin{corollary}
Let $X$ be a Du Val del Pezzo surface.
\begin{enumerate}
\item
If $K_X^2=1$ and $\Aut(X)$ is infinite, then $\uprho(X)=1$.
\item
If $K_X^2>1$ and $\uprho(X)=1$, then $\Aut(X)$ is infinite.
\item
If $K_X^2\geqslant 6$ or $K_X^2=5$ and $X$ is singular, then $\Aut(X)$ is infinite.
\end{enumerate}
\end{corollary}

Many particular parts of our classification have been previously studied from different perspectives.
For examples, the~Du Val~del Pezzo surfaces
admitting an effective action of the~group $\Ga^2$ and $\Ga \rtimes \Gm$ have been classified in \cite{Derenthal2010,Derenthal-Loughran}.
The classification of toric Du Val del Pezzo surfaces is well-known for specialists (see e.g. \cite{Poonen}).
Du Val del Pezzo surfaces that admit a~faithful action of the~group $\Gm$ have been studied in \cite{Arz-De-Hau-Laf,Hausen2011,Herppich2014} in terms of their Cox rings.
Moreover, when we were finishing the~final version of this paper, we were informed that Main Theorem
has been independently proven in \cite{MartinStadlmayr} using completely different approach, which also works in positive characteristic.

Note that the~complete classification of \textit{all} Du Val~del Pezzo surfaces have been known for a long time \cite{DuVal,Dolgachev}.
The basic problem is that its is very huge and to choose surfaces with infinite automorphism group
typically takes a lot of efforts.

\begin{remark}
\label{remark:ambiguity}
Almost all surfaces in Big Table are explicitly given by their defining equations,
since they are not always uniquely determined by their degree and singularities.
For~example, the~cubic surface in $\PP^3$ given by
$$
x_3x_0^2+x_1^3+x_2^3+x_0x_1x_2=0
$$
has one singular point of type $\type{D}_4$, its automorphism group is finite,
and it is not isomorphic to the~cubic surface~\ref{d=3:D4}, which has the~same singularity.
Similarly, the~quartic surface in the~weighted projective space $\PP(1,1,1,2)$ that is given by the~equation
$$
y_2^2=y_1^3y_1''+y_1'^4+y_1'^2y_1^2
$$
is a~del Pezzo surface of degree $2$ that has one singular point of type $\type{E}_6$.
It is not isomorphic to the~del Pezzo surface~\ref{d=2:E6}, which has the~same degree and the~same singularity.
There~are more examples like this:
the~surface~\ref{d=1:E8} and the~sextic surface in $\PP(1,1,2,3)$ given by
$$
y_3^2=y_2^3+y_1'y_1^5+y_2^2y_1^2
$$
are the~only del Pezzo surfaces of degree $1$ with singular point of type $\type{E}_8$.
They are not isomorphic.
In fact, the~latter surface is the~only Du Val del Pezzo surface whose class group~is~$\ZZ$ that does not appear in our Big Table (see Remark~\ref{remark:d-1-E8}).
\end{remark}

Let us briefly describe the~structure of this paper.
In Section~\ref{section:preliminaries}, we present several basic facts about Du Val del Pezzo surfaces
which are used in the~proof of Main Theorem.
In Section~\ref{section:Fano--Weil}, we prove Theorem~\ref{theorem:index>1}, which together with Main Theorem imply

\begin{corollary}
\label{corollary:ind-Aut}
Let $X$ be a~Du Val del Pezzo surface with $K_X^2\geqslant 3$, and~let $\ind(X)$ be its Fano--Weil index.
Suppose that $\ind(X)>1$. Then $\Aut(X)$ is infinite.
\end{corollary}

In Section~\ref{section:dP-large-degree1},
we prove Main Theorem for del Pezzo surfaces of degree $\geqslant 4$.
Then, in Sections~\ref{section:deg-1}, \ref{section:deg-2},~\ref{section:cubic-surfaces}, we prove Main Theorem for del Pezzo surfaces of degree $1$, $2$, $3$, respectively.
In Section~\ref{section:tables}, we present Big Table.
In Appendix~\ref{section:appendix}, we describe lines on del Pezzo surfaces that appear in Big Table
together with the~dual graphs of the~curves with negative self-intersection numbers on their minimal resolutions.
Finally, in Appendix~\ref{section:appendixb},
we will recall classification of Du Val del Pezzo surfaces whose Weil divisor class group is cyclic,
and present an alternative proof of Main Theorem for them.

\subsection*{Notations.}
Throughout this paper, we will use the~following notation:
\begin{itemize}
\item
$\mumu_n$ is a~cyclic subgroup of order $n$.
\item
$\Ga$ is a~one-dimensional unipotent additive group.
\item
$\Gm$ is a~one-dimensional algebraic torus.
\item
$\BB_n$ is a~Borel subgroup of $\PGL_n(\Bbbk)$.
\item
$\UU_n$ is a~maximal unipotent subgroup of $\PGL_n(\Bbbk)$.
\item
$\Ga\rtimes_{(n)} \Gm$ is a~semidirect product $\Ga$ and $\Gm$ such that $\Gm$ acts on $\Ga$ as $\mathbf x \mapsto t^n \mathbf x$.
This~group is isomorphic to the~following group:
$$
\left\{\left.\left(
\begin{array}{ccc}
t^r&0 \\
a&t^{s} \\
\end{array}
\right)\in\GL_{2}(\Bbbk)\ \right|\ t\in\Bbbk^{\ast}\ \text{and}\ a\in\Bbbk\right\},
$$
where $n=s-r$. Indeed, the required isomorphism follows from
$$
\left(\begin{array}{ccc}
t^r&0 \\
0&t^{s} \\
\end{array}
\right)
\left(\begin{array}{ccc}
1&0 \\
a&1 \\
\end{array}
\right)
\left(\begin{array}{ccc}
t^r&0 \\
0&t^{s} \\
\end{array}
\right)^{-1}=
\left(\begin{array}{ccc}
1&0 \\
at^{s-r}&1 \\
\end{array}
\right).
$$
Observe that~$\Ga\rtimes_{(0)} \Gm=\Ga\times\Gm$, $\Ga\rtimes_{(1)} \Gm=\BB_2$ and
$$
\Ga\rtimes_{(n)} \Gm\cong \Ga\rtimes_{(-n)} \Gm.
$$
Therefore, we will always assume that $n\geqslant 0$.
If~$n>0$, the~center of $\Ga\rtimes_{(n)} \Gm$ is $\mumu_n$.
This implies that $\Ga\rtimes_{(n_1)} \Gm\cong\Ga\rtimes_{(n_2)} \Gm$ $\iff$ $n_1=\pm n_2$.

\item
$\FF_n$ is the~Hirzebruch surface.
\item
$\PP(a_1,\dots, a_n)$ is the~weighted projective space.
\item
For a~weighted projective space $\PP(a_0,a_1,\ldots,a_n)$, we denote by $y_{a_0},y_{a_1},\ldots,y_{a_n}$ the~coordinates on it of
weights $a_0,a_1,\ldots,a_n$, respectively.
% For a~product of weighted projective spaces, we denote the~coordinates on the~first factor by $x_0,x_1,\ldots,x_n$,
% we denote the~coordinates on the~second factor by $y_0,y_1,\ldots,y_m$,
% and we denote the~coordinates on the~third factor (if any) by $z_0,z_1,\ldots,z_k$.

\item
For a~variety $X$, we denote by $\Sing(X)$ the~set of its singular points.
\item
For a~variety $X$, we denote by $\uprho(X)$ the~rank of the~Weil divisor class group $\Cl(X)$.
\item
For a~variety $X$ and its (possibly reducible) reduced subvariety $Y\subseteq X$, $\Aut(X,Y)$ denotes the~group consisting of all automorphisms  in $\Aut(X)$ that maps $Y$ into itself.

\item
For a~surface $X$ with Du Val singularities, $\Type(X)$ denotes the~type of its singularities.
If~$\Type(X)=\type{D}_42\type{A}_1$, then $\Sing(X)$ consists of a~point of type $\type{D}_4$,
and $2$ points of type~$\type{A}_1$.
\item
For a~Du Val del Pezzo surface $X$, $\ind(X)$ denotes its Fano--Weil index, which is defined as follows:
$$
\ind(X)=\max\Big\{ t\in \ZZ \ \big\vert\ -K_X\sim tA,\ \text{where $A$ is a~Weil divisor on $X$}\Big\}.
$$
\end{itemize}

\subsection*{Acknowledgments.}
This work was supported by the~Royal Society grant No. IES\textbackslash R1\textbackslash 180205 and by the~Russian Academic Excellence Project 5-100.

The authors would like to thank the anonymous referee for many valuable advices that helped to improve this paper.

\section{Del Pezzo surfaces with Du Val singularities}
\label{section:preliminaries}

Let $X$ be a~Du Val del Pezzo surface with $d:=K_X^2$.
Then $d$ is known as the~degree of the~surface~$X$.
Let $\mu\colon\widetilde{X}\to X$ be the~minimal resolution of singularities.
Then
$$
K_{\widetilde{X}}\sim\mu^*K_X,
$$
so that $\widetilde{X}$ is a~\emph{weak del Pezzo surface}, that is, the~anticanonical divisor $-K_{\widetilde X}$ is nef and big.
By~the~Noether formula $d=10-\uprho(X)\leqslant 9$ and by the~genus formula
every irreducible curve on $\widetilde{X}$ with negative \mbox{self-intersection} number is either $(-1)$ or $(-2)$-curve.
Moreover, one of the~following holds (see \cite{Brenton1980,Hidaka1981}):
\begin{enumerate}
\item
$K_X^2=9$ and $\widetilde{X}\cong X\cong\PP^2$;
\item
$K_X^2=8$ and $\widetilde{X}\cong X\cong\FF_1$;
\item
$K_X^2=8$ and $\widetilde{X}\cong X\cong\PP^1\times\PP^1$;
\item
$K_X^2=8$, $\widetilde{X}\cong\FF_2$ and $X$ is a~quadric cone in $\PP^3$;
\item
$K_X^2\leqslant 7$ and there exists a~$\Aut^0(X)$-equivariant diagram
\begin{equation}
\label{equation:min-resolution}
\vcenter{
\xymatrix@R=0.8em{
&\widetilde{X}\ar[dr]^{\mu}\ar[dl]_{\varphi}&\\
\PP^2&&X.}}
\end{equation}
where $\varphi$ is a~suitable contraction of $(-1)$-curves.
\end{enumerate}
Moreover, it follows from the~Kawamata--Viehweg vanishing and the~exponential exact sequence that the~group $\Pic(X)$ is torsion free.

\begin{corollary}
\label{corollary:connected-group-rational-surface}
Let $G$ be a~connected algebraic subgroup in $\Aut(X)$.
Suppose that $d=K_X^2\leqslant 7$. Then $G$ is isomorphic to a~subgroup of the~following group:
$$
\left \{\left.\begin{pmatrix}
a_{11}&0&0\\
a_{21}&a_{22}&0\\
a_{31}&a_{32}&1
\end{pmatrix}\in\GL_3(\Bbbk)\ \right|\ a_{ij}\in\Bbbk, a_{11}\ne 0, a_{22}\neq 0\right\}\cong\UU_3\rtimes\Gm^{2}.
$$
In particular, the~group $G$ is solvable.
If $G$ is reductive and non-trivial, then $G\cong\Gm$~or~$G\cong\Gm^2$.
Similarly, if $G$ is unipotent, then $G\cong\Ga$ or $G\cong \Ga^2$ or $G\cong\UU_3$.
\end{corollary}

\begin{proof}
This follows from the~fact that the~diagram~\eqref{equation:min-resolution} is $G$-equivariant.
\end{proof}

\begin{example}[{\cite[Proposition~8.1]{Coray1988}}]
\label{example:dP7-A1}
Suppose that $d=7$ and $X$ is singular. The surface $X$~is~unique.
The morphism $\varphi$ is the~blow up of two points, $\widetilde{X}$ contains unique $(-2)$-curve, and $\Type(X)=\type{A}_1$.
The surface $X$ contains two $(-1)$-curves. The dual graph of curves with negative self-intersection numbers
has the~form
$$
\xymatrix{
\circ\ar@{-}[r]&\bullet\ar@{-}[r]&\bullet}
$$
where $\bullet$ denotes a~$(-1)$-curve, and $\circ$ denotes the~$(-2)$-curve. Using \eqref{equation:min-resolution}, we see that
$$
\Aut^0(X)\cong \Aut(\PP^2,\ell, P)\cong \BB_3,
$$
where $\ell$ is a~line on $\PP^2$, and $P$ is a~point in $\ell$. Note that $X$ is the~surface~\ref{d=7}.
\end{example}

The~type $\Type(X)$ does not always determine the~dual graph of
curves with negative self-intersection numbers~in~$\widetilde{X}$.
However, this graph is always determined by the~type $\Type(X)$ and the~number of $(-1)$-curves~in~$\widetilde{X}$.
In the~following, we denote by $\numl(X)$ the~number of \mbox{$(-1)$-curves} in~the~surface $\widetilde{X}$.

\begin{example}
\label{example:dP6-A1}
If $d=6$ and $\Type(X)=\type{A}_1$, then $X$ is one of the~surfaces~\ref{d=6:A1-3l} or~\ref{d=6:A1-4l} in Big~Table.
If $X$ is the~surface~\ref{d=6:A1-3l}, then $\numl(X)=3$. On the~other hand, if $X$ is the~surface~\ref{d=6:A1-4l}, then $\numl(X)=4$.
The dual graph of curves with negative self-intersection numbers in $\widetilde{X}$ is given in Appendix~\ref{section:appendix}.
\end{example}

Using the~Riemann--Roch formula and Kawamata--Viehweg vanishing, we get $\dim |-K_X|=d$.
Let $\Phi\colon X\dasharrow\PP^d$ be the~rational map given by $|-K_X|$.
The linear system $|-K_X|$ does not have fixed components, and it contains a~smooth elliptic curve (see \cite{Demazure1980}).
Using this fact, one can prove

\begin{theorem}[{\cite{Demazure1980,Hidaka1981}}]
\label{theorem:delPezzo}
The following assertions hold:
\begin{enumerate}
\item
if $d\geqslant 2$, then $|-K_X|$ is base point free, so that $\Phi$ is a~morphism;
\item
if $d\geqslant 3$, then $-K_X$ is very ample, so that $\Phi$ is an embedding;

\item
if $d=3$, then $\Phi(X)$ is a~cubic surface in $\PP^3$;

\item
if $d\geqslant 4$, then $\Phi(X)$ is an intersection of quadrics in $\PP^{d}$;

\item
if $d=2$, then $\Phi$ is a~double cover that is branched over a~possibly reducible quartic curve,
so~that $X$ is a~hypersurface in $\PP(1,1,1,2)$ of degree $4$;

\item
if $d=1$, then $|-K_X|$ is an elliptic pencil,
its base locus consists of one point $O\notin\Sing(X)$,
and every curve in $|-K_X|$ is irreducible and smooth at $O$;

\item
if $d=1$, then $|-2K_X|$ defines a~double cover $X\to \PP(1,1,2)$ branched over \mbox{a~sextic~curve},
so~that $X$ is a~hypersurface in $\PP(1,1,2,3)$ of degree~$6$.
\end{enumerate}
\end{theorem}

The number of $(-1)$-curves in $\widetilde{X}$ is finite.

\begin{definition}
\label{def:line}
An irreducible curve $L\subset X$ is a~\emph{line} if $L=\mu(\widetilde{L})$ for a~$(-1)$-curve $\widetilde{L}\subset\widetilde{X}$.
\end{definition}

If $d\geqslant 3$, then lines in $X$ are usual (projective) lines in $\Phi(X)\subset \PP^{d}$.
Conversely, if $d\geqslant 3$, then lines in $\Phi(X)$ are lines in the~sense of Definition~\ref{def:line}.
Moreover, if $d=2$, then lines in $X$ are smooth rational curve.
Furthermore, if $d=1$ and $L$ is a~line in $X$, then
\begin{itemize}
\item either $L$ is singular curve in $|-K_X|$ such that $\Sing(L)\subset \Sing(X)$,
\item or $L$ is a~smooth rational curve that does not contain the~base point of the~pencil $|-K_X|$.
\end{itemize}
Note that $\numl(X)$ is the~number of lines in $X$. Then $\numl(X)>0$ unless $X$ is $\PP^2$, $\PP^1\times\PP^1$ or $\PP(1,1,2)$.

\begin{lemma}
\label{delPezzo:poin-lines}
Assume that $d\geqslant 3$.
Let $P$ be a~point in $X$, and let $\numl(X,P)$ be the~number of lines in $X$ passing through $P$.
\begin{enumerate}
 \item
If $P\in X$ is smooth, then
\[
\numl(X,P)\leqslant
\begin{cases}
3 & \text{if $d= 3$,}\\
2 &  \text{if $d\geqslant 4$.}
\end{cases}
\]
 \item
If $P\in X$ is singular, then
\[
\numl(X,P)\leqslant
\begin{cases}
6 & \text{if  $d=3$,}\\
4 & \text{if  $d=4$,}\\
3 &\text{if  $d\geqslant 5$.}
\end{cases}
\]
\end{enumerate}
\end{lemma}

\begin{proof}
Let $L_1,\dots,L_r$ be all the lines on $X$ passing through $P$, and let $\mathbb{T}_{P,X}\subset \PP^d$ be the embedded tangents space to $X$ at the point $P$.
Then
$$
\bigcup_{i=1}^{r} L_i\subseteq X\cap \mathbb{T}_{P,X}.
$$
If $P\in X$ is a smooth point, then $\dim \mathbb{T}_{P,X}=2$, so that $r\leqslant d$.
Moreover, if $d\geqslant 4$, then $r\leqslant 2$, because $X$ is an intersection of quadrics in this case.
Thus, we may assume that $P$ is a~singular point of the surface $X$.
Then $\dim \mathbb{T}_{P,X}=3$, since $P$ is a Du Val singular point of the surface~$X$.
Hence, if $d\geqslant 4$, then $r\leqslant 4$, because $X$ is an~intersection of quadrics in this case.

Suppose that $d=3$. We may assume that $P=(0:0:0:1)$. Then $X$ is given in $\mathbb{P}^3$ by
$$
x_3q_2(x_0,x_1,x_2)+q_3(x_0,x_1,x_2)=0,
$$
where $q_2$ and $q_3$ are homogeneous forms of degree $2$ and $3$, respectively.
Then the (set-theoretic) union of the lines $L_1,\ldots,L_r$ is given by the system of equations $q_2=q_3=0$, so that $r\leqslant 6$.

To complete the proof, we may assume that $d\geqslant 5$.
We only consider the case $d=5$, since the~proof is similar in the~remaining cases.
Let us show that $r\leqslant 3$.
To do this, suppose that~$r\geqslant 4$. Let us seek for a~contradiction.

Let  $Q$ be a~point in $X$ that is not contained in any line in $X$ (it exists since $\numl(X)$ is finite).
Keeping in mind that the~Zariski tangent space of the~surface $X$ at the~point $P$ is three-dimensional,
we conclude that there exists a~hyperplane $H$ in $\PP^5$ that contains the~lines $L_1$, $L_2$, $L_3$, $L_4$ and the~point $Q$.
Then
$$
H\big\vert_{X}=C+\sum_{i=1}^{4} L_i,
$$
where $C$ is a~curve in $X$ that passes through $Q$. Counting degrees, we see that $\deg(C)\leqslant 1$,
so that $C$ is a~line, which contradicts the~choice of the~point $Q$.
\end{proof}

\begin{lemma}
\label{lemma:lines}
Suppose that $d\leqslant 7$. For any singular point of $X$ there is a~line passing through it.
\end{lemma}

\begin{proof}
The required assertion follows from the~existence of the~diagram~\eqref{equation:min-resolution}.
\end{proof}

Since the~Du Val singularities are $\QQ$-factorial, $\uprho(X)$ is equal to the~rank of~the~Weil divisor class group $\Cl(X)$.
\begin{lemma}
\label{lemma:Cl}
Suppose that $d\leqslant 7$. Then the~following assertions hold.
\begin{enumerate}
\item
\label{lemma:Cl:gens}
The group $\Cl(X)$ is generated by the~classes of lines in $X$.

\item
\label{lemma:Cl:tors}
Let $\Cl(X)_{\tors}\subset \Cl(X)$ be the~torsion subgroup and let $n$ be the~order of the~group $\Cl(X)_{\tors}$.
There is a~Galois abelian cover $\pi\colon Y\to X$ of degree $n$
which is \'etale outside of $\Sing(X)$, where $Y$ is a~Du Val del Pezzo surface such that
$$
K_{Y}^2=d\, n,
$$
so that $n\leqslant \frac{9}{d}$.

\item
\label{lemma:Cl:tors-l}
If $\uprho(X)=1$ and $X$ contains two distinct lines $L$ and $L^\prime$, then $L\not\sim L^\prime$ and $L\simQ L^\prime$.

\item
\label{lemma:Cl:extr}
Every extremal ray of the~Mori cone $\NE(X)$ is generated by the~class of a~line.
\item
\label{lemma:Cl:convex}
For every effective divisor $D\in\Cl(X)$, there are $a_0,a_1,\ldots,a_r\in\ZZ_{\geqslant 0}$ such that
$$
D\sim a_0(-K_X)+\sum_{i=1}^{r}a_iL_i
$$
where $L_1,\ldots,L_r$ are lines in $X$, $r=\numl(X)$, and $a_0=0$ if $d\ne 1$.
\end{enumerate}
\end{lemma}

\begin{proof}
The assertion~\ref{lemma:Cl:tors} follows from a well-known construction, see e.g.~\cite[\S~3.6]{YPG}.

To prove the assertion \ref{lemma:Cl:tors-l}, observe that
$$
L\simQ L^\prime,
$$
because the numerical and $\QQ$-linear equivalences on the surface $X$ coincide.
But $L\not\sim L^\prime$, because otherwise $X$ would contain a pencil of lines, which contradicts $\numl(X)<\infty$.

The assertion \ref{lemma:Cl:extr} follows from Lemma~\ref{extremal-contractions} below.

Let us prove the assertion~\ref{lemma:Cl:convex}.
Let $\widetilde{D}$ be the~proper transform on $\widetilde{X}$ of the~divisor~$D$.
To~prove~\ref{lemma:Cl:convex}, it~is~sufficient to show that the~divisor $\widetilde{D}$ is rationally equivalent to a~convex integral
linear combination of $(-1)$ and $(-2)$-curves (and $-K_{\widetilde{X}}$ if $d=1$).

We may assume that $\widetilde{D}$ is an irreducible curve.

Let us use induction on $\dim |\widetilde{D}|$.
If $\dim |\widetilde{D}|=0$, then $\widetilde{D}$ is either a~$(-1)$ or $(-2)$-curve
by the~Riemann-Roch formula and Kawamata-Viehweg vanishing.
This is the~base of induction.

Suppose that $\dim |\widetilde{D}|\geqslant 1$ and the~required assertion holds for any effective
divisor $\widetilde{D}^\prime$ on $\widetilde{X}$ such that $\dim|\widetilde{D}^\prime|<\dim |\widetilde{D}|$.
Observe that $\widetilde{D}$ is nef.
Thus, if $\widetilde{D}^2=0$, then $|\widetilde{D}|$ is base point free and gives a~conic bundle $\widetilde{X}\to\PP^1$,
which must have at least one reducible fiber, because $\uprho(\widetilde{X})\geqslant 3$.
Hence, if $\widetilde{D}^2=0$, then we can proceed by induction.
Thus, we may assume that $\widetilde{D}^2\geqslant 1$.

If $\widetilde{D}$ is not ample, then $\widetilde{X}$ contains an irreducible curve $\widetilde{C}$ such that $\widetilde{D}\cdot \widetilde{C}=0$,
which implies that $\widetilde{C}^2<0$ by the~Hodge index theorem, so that $\widetilde{C}^2=-1$ or $\widetilde{C}^2=-2$,
which gives
$$
\dim |\widetilde{D}-\widetilde{C}|\geqslant\dim |\widetilde{D}|-1\geqslant 0.
$$
Hence, if $\widetilde{D}$ is not ample, then there a~exists an effective divisor $\widetilde{D}^\prime$ such that $\widetilde{D}\sim \widetilde{D}^\prime+\widetilde{C}$,
so that we can proceed by induction.
Therefore, we may assume that $\widetilde{D}$ is ample.

Suppose that $\widetilde{D}\sim -K_{\widetilde{X}}$ and $K_{\widetilde{X}}^2\geqslant 2$.
Then for any $(-1)$-curve $\widetilde{C}$ on $\widetilde{X}$ we have
$$
\dim |\widetilde{D}-\widetilde{C}|\geqslant \dim |\widetilde{D}|-2\geqslant 0,
$$
so that there is an effective divisor $\widetilde{D}^\prime$ such that $\widetilde{D}\sim \widetilde{D}^\prime+\widetilde{C}$,
and we can proceed by induction.

Finally, we assume that $\widetilde{D}$ is ample and $\widetilde{D}\not\sim -K_{\widetilde{X}}$.
There is $a\in\NN$ such that $\widetilde{D}+aK_{\widetilde{X}}$ is nef but not ample,
because the~Mori cone of the~surface $\widetilde{X}$ is generated by $(-1)$-curves and $(-2)$-curves.
Now using~the~Riemann-Roch formula and Kawamata--Viehweg vanishing, we see that the~linear system $|\widetilde{D}+aK_{\widetilde{X}}|$~contains a~divisor $\widetilde{D}^\prime$,
so that
$$
\widetilde{D}\sim \widetilde{D}^\prime-aK_{\widetilde{X}},
$$
where $\widetilde{D}^\prime$ and $-K_{\widetilde{X}}$ are both decomposable in the~required form.
\end{proof}

\begin{corollary}
\label{corollary:Cl}
One has $\numl(X)\geqslant\uprho(X)$.
Moreover, if $\numl(X)=\uprho(X)$, then $\Cl(X)$ is torsion free,
and every line in $X$ generates an extremal ray of the~Mori cone $\NE(X)$.
\end{corollary}

\begin{corollary}
\label{corollary:Ga}
Suppose that $d\leqslant 7$, and $X$ admits a~faithful $\Ga^2$-action.
Then $\numl(X)=\uprho(X)$.
Moreover, the~complement to the~open orbit coincides with the~union of lines.
\end{corollary}

\begin{proof}
Let $U$ be the~open $\Ga^2$-orbit, let $\widetilde{U}=\mu^{-1}(U)$, let $\overline{U}=\varphi(\widetilde U)$,
let $B=X\setminus U$, let $\widetilde{B}=\widetilde X\setminus \widetilde{U}$, and let $\overline{B}=\PP^2\setminus\overline{U}$.
Then
$$
U\cong\widetilde{U}\cong\overline{U}\cong \Aff^2,
$$
so that the~curve $\overline{B}$ must be a~line.
Then $\widetilde{B}$ has $\uprho(\widetilde X)$ components, and $B$ has $\uprho(X)$ components.
Since all $(-1)$-curves on $\widetilde X$ are contained in $\widetilde{B}$,
we see that all the~lines in $X$ are contained in $B$. This gives $\numl(X)\leqslant\uprho(X)$.
But $\numl(X)\geqslant \uprho(X)$ by Corollary~\ref{corollary:Cl}.
\end{proof}

Observe that a~line $L$ on the~surface $X$ generates an extremal ray of~$\NE(X)$ $\iff$ $L^2\leqslant 0$.

\begin{lemma}[{\cite[Proposition~1.2]{Morrison-1985},\cite[\S~7.1]{Prokhorov-2001}}]
\label{extremal-contractions}
Let $V$ be a~surface that has Du Val singularities, and let $\psi\colon V\to Y$ be an extremal Mori contraction.
Then one of the~following holds:
\begin{enumerate}
\item
\label{extremal-contractions2}
either $\psi$ is a~weighted blow up of a~smooth point in $Y$ with weights $(1,n)$,
the~exceptional curve $E$ is smooth and rational, one has $E^2=-\frac1{n}$,
and $E\cap \Sing(V)$ consists of one  point which is  of type $\type{A}_{n-1}$ on $V$;
\item
\label{extremal-contractions1}
or $\psi$ is a~conic bundle, one has $-K_V\cdot F=2$ and $F_{\red}\cong\PP^1$ for any its scheme fiber~$F$,
and if $F$ is not reduced, then one of the~following three cases holds:
\begin{itemize}
\item
$F\cap\Sing(V)$ consists of two singular points of type $\type{A}_1$;
\item
$F\cap\Sing(V)$ consists of one singular point of type $\type{A}_3$;
\item
$F\cap\Sing(V)$ consists of one singular point of type $\type{D}_n$, where $n\geqslant 4$.
\end{itemize}
\end{enumerate}
In the~case~\ref{extremal-contractions2}, we say that $\psi$ is a~\emph{$(1,n)$-contraction}.
\end{lemma}

Applying this lemma to our Du Val del Pezzo surface $X$, we get

\begin{corollary}
\label{corollary:negative-curves-DP}
Let $E$ be an irreducible curve on $X$ such that $E^2<0$.
Then $E$ is a~line on $X$, and $E$ is an exceptional divisor of a~$(1,n)$-contraction for some $n\geqslant 1$.
\end{corollary}

\begin{corollary}
\label{corollary:extremal-contractionsDP}
Suppose there exists a~birational morphism $\psi\colon X\to Y$ that is a~$(1,n)$-contraction,
and let $E$ be the~exceptional curve of the~morphism $\psi$. Then
\begin{itemize}
\item
the~point $\psi(E)$ is a~smooth point of the~surface $Y$;
\item
$Y$ is a~Du Val del Pezzo surface, $K_Y^2=d+n$ and $\uprho(Y)=\uprho(X)-1$;
\item
the~point $\psi(E)$ is not contained in a~line in $Y$.
\end{itemize}
\end{corollary}

\begin{corollary}
\label{corollary:extremal-contractions-faces}
Let $\psi\colon X\to Y$ be a~contraction of a~proper face of the~cone $\NE(X)$. Then
\begin{itemize}
\item
either the~morphism $\psi$ is birational, $Y$ is a~Du Val del Pezzo surface, and $\psi$ contracts a~disjoint union of lines on the~surface $X$,
\item
or the~morphism $\psi$ is a~conic bundle and $Y\cong\PP^1$.
\end{itemize}
\end{corollary}

If~a del Pezzo surface $X$ is smooth and $\uprho(X)\geqslant 2$, then $X$ always admits a~conic bundle contraction.
However, this is not always the~case if  $X$ has Du Val singularities.

\begin{lemma}
\label{lemma:conic-bundles}
Let $X$ be a Du Val del Pezzo surface of degree $d$ with $\uprho(X)\geqslant 2$.
\begin{enumerate}
\item\label{conic:bundle:3}
Assume that $d=3$. Then there exists a~conic bundle structure $\psi\colon X\to \PP^1$
if and only if $X$ contains a~line $L$ that is contained in $X\setminus\Sing(X)$.

\item\label{conic:bundle:4}
Assume that $d=4$. Then there exists a~conic bundle structure $\psi\colon X\to \PP^1$
if and only if there is a~double cover $\pi\colon X\to\PP^1\times\PP^1$ branched over a~curve of degree $(2,2)$.
\end{enumerate}
\end{lemma}

\begin{proof}
If $d=3$, then $X$ is a~cubic surface in $\PP^3$, so that every conic bundle $\psi\colon X\to \PP^1$
is given by the~linear projection from some line in $X$ that does not contain singular points of the~surface $X$,
so that $X$ admits a~conic bundle contraction if and only if such a line exists.

Assume that $d=4$.
If there is a~double cover $X\to\PP^1\times\PP^1$ branched over a~curve of degree~$(2,2)$,
then composing it with a~projection to one of the~factors, we obtain the~required~conic bundle.
Thus, we may assume that there exists a~conic bundle $\psi\colon X\to \PP^1$.
Let~$C$ be its general fiber.
Then $|-K_X-C|$ is base point free and gives another conic bundle $\psi^\prime\colon X\to\PP^1$.
Let $\pi=\psi\times \psi^\prime$.
Then $\pi\colon X\to\PP^1\times\PP^1$ is the~required double cover.
\end{proof}

\section{The~Fano--Weil index of Du Val del Pezzo surfaces}
\label{section:Fano--Weil}

Recall from Section~\ref{section:intro} that $\ind(X)$ is the~Fano--Weil index of a~del Pezzo surface $X$.

\begin{lemma}
\label{lemma:rho=1=>index>1}
Let $X$ be a Du Val del Pezzo surface with  $\uprho(X)=1$ and $K_X^2\geqslant 3$. Then $\ind(X)=K_X^2$.
\end{lemma}

\begin{proof}
Let $d:=K_X^2$ and $D:=K_X+dL$, where $L$ is a~line on $X$.
If $D\sim 0$, then we are done.
Thus, we may assume that~$D\not\sim 0$.
Since~$D\simQ 0$, the~divisor $D$ is a~non-trivial torsion in $\Cl(X)$.
Let $n$ be its order. Then
$$
2\leqslant n\leqslant \frac{9}{d}
$$
by Lemma~\ref{lemma:Cl}, so that either $d=3$ or $d=4$.

Suppose that $d=4$. Then either $\Type(X)=\type{A}_32\type{A}_1$ or $\Type(X)=\type{D}_5$ by \cite[Proposition~6.1]{Coray1988}.
In~the~former case, we see that $4L$ is a~Cartier divisor, so that $D\sim 0$, since $\Pic(X)$ is torsion free.
In the~latter case, $L$ is the~unique line in $X$ by \cite[Proposition~6.1]{Coray1988}, so that $D\sim 0$ by Lemma~\ref{lemma:Cl}.
Thus, in both cases we obtain a~contradiction with our assumption that $D\not\sim 0$.

Thus, we see that $d=3$. Then either $n=2$ or $n=3$. If $n=2$, we have
$$
K_X+3(L+D)\sim 4K_X+12L\sim 4(K_X+3D)\sim 0,
$$
so that $\ind(X)=3$. If $n=3$, then $X\cong\PP^2/\mumu_3$ by Lemma~\ref{lemma:Cl}, which implies that $\Type(X)=3\type{A}_2$.
In this case, the~divisor $3L$ is Cartier, which gives $D\sim 0$, because $\Pic(X)$ is torsion free.
\end{proof}

The number $\ind(X)$ divides the~degree $d$ of the~del Pezzo surface $X$, so that $\ind(X)=1$ if $d=1$.
If~$d\geqslant 2$, then the~Fano--Weil index $\ind(X)$ is closely related to the~following notion:

\begin{definition}
\label{definition:weakly-minimal}
A del Pezzo surface $X$ is said to be \emph{weakly minimal}
if $X$ does not contain lines that are contained in the~smooth locus of the~surface $X$.
\end{definition}

\begin{remark}
\label{remark:min}
If  $X$ is a weakly minimal Du Val del Pezzo surface, and $\psi\colon X\to Y$ is a~birational contraction,
then $Y$ is also a~weakly minimal Du Val del Pezzo surface by Corollary~\ref{corollary:extremal-contractionsDP}.
\end{remark}

\iffalse
Arguing as in the~proofs of Lemma~\ref{lemma:conic-bundles}\ref{conic:bundle:4} and Theorem~\ref{theorem:Fano-Weil-index-big}, we obtain

\begin{corollary}
\label{corollary:quartic-weakly-minimal}
Suppose that $d=4$, and there is a~double cover $\pi\colon X\to\PP^1\times\PP^1$ branched over a~curve $B$ of degree $(2,2)$.
Then $\pi$ is $\Aut^0(X)$-equivariant, and the~following are~equivalent:
\begin{enumerate}
\item
the~curve $B$ is reducible and has no irreducible component of degree $(1,2)$ or $(2,1)$;
\item
$\ind(X)>1$.
\end{enumerate}
\end{corollary}
\fi
Now, we prove the~following result.

\begin{proposition}
\label{proposition:ind}
Let $X$ be a Du Val del Pezzo surface and let $d:=K_X^2$.
\begin{enumerate}
\item
\label{proposition:ind=d}
If $\ind(X)=d$, then $d\leqslant 6$ and $X$ is a~hypersurface in $\PP(1,2,3,d)$ of degree $6$ given by
$$
y_3^2+ y_2^3+\lambda_1 y_1^4y_2+\lambda_2 y_1^6 +y_d\phi(y_1,y_2,y_d)=0,
$$
where $\phi$ is a~polynomial of degree $6-d$, and $\lambda_1$ and $\lambda_2\in\Bbbk$ such that $4\lambda_1^3+27\lambda_2^2\neq 0$,
\item\label{proposition:ind=d/2}
If $\ind(X)=\frac{d}2$, then $X$ is a~hypersurface in $\PP(1,1,2,\frac{d}2)$ of degree $4$ given by
\begin{equation*}
\label{eq:index=d/2}
y_2^2+y_1y_1'(y_1-y_1')(y_1-\lambda y_1')+x_3\phi(y_1,y_1',x_3)=0,
\end{equation*}
where $\phi$ is a~polynomial of degree $4-\frac{d}2$, and $\lambda\in\Bbbk$ such that $\lambda\ne 0$ and $\lambda\ne 1$.
\item
\label{proposition:ind=d/3}
If $d=6$ and $\ind(X)=2$, then $X$ is a~hypersurface in $\PP(1,1,1,2)$ of degree $3$ given by
\begin{equation*}
\label{eq:index=d/3}
\psi(y_1,y_1',y_1'') +y_2\phi(y_1,y_1',y_1'')=0,
\end{equation*}
where $\psi$ and $\phi$ are polynomials of degree $3$ and $1$, respectively.
\end{enumerate}
\end{proposition}

\begin{proof}
Let us only prove the~assertion~\ref{proposition:ind=d}, since the~assertions~\ref{proposition:ind=d/2} and~\ref{proposition:ind=d/3} can be proved similarly.
Let $C$ be a~general curve in $|-K_X|$. Then $C$ is a~smooth elliptic curve (see, for example, \cite{Demazure1980}).
Suppose that $\ind(X)=d$. Then $-K_X\sim dA$, where $A$ is a~Weil divisor on $X$.
Consider the~natural homomorphism of graded algebras
$$
\Phi:\mathrm{R}(X,A):=\bigoplus_{n\geqslant 0} H^0\Big(X,\OOO_X\big(nA\big)\Big)\longrightarrow \bigoplus_{n\geqslant 0} H^0\Big(C,\OOO_C\big(nA\big)\Big)=:\mathrm{R}(C,A)
$$
By the~Kawamata--Viehweg vanishing it is surjective. Note that $\OOO_C(A)$ is a~line bundle of degree~$1$.
It is well-known that $\mathrm{R}(C,A)$ is generated by $3$ elements $\bar{y}_1$, $\bar{y}_2$, $\bar{y}_3$ with $\deg \bar{y}_i=i$ such that
$$
\bar{y}_3^2+\bar{y}_2^3+ \lambda_1 \bar{y}_1^4\bar{y}_2 +\lambda_2 \bar{y}_1^6=0
$$
for some $\lambda_1$ and $\lambda_2$ in $\Bbbk$ such that $4\lambda_1^3+27\lambda_2^2\neq 0$.
The kernel of $\Phi$ is generated by a~homogeneous element $y_d$ of degree $d$.
Take arbitrary elements $y_1$, $y_2$ and $y_3$ in $\mathrm{R}(X,A)$ such that $\Phi(y_i)=\bar{y}_i$.
Then $\mathrm{R}(X,A)$ is generated by $y_1$, $y_2$, $y_3$ and $y_d$. This gives us an embedding
$$
X\cong\operatorname{Proj}\Big(\mathrm{R}\big(X,A\big)\Big) \hookrightarrow \operatorname{Proj}\Big(\Bbbk\big[y_1,y_2,y_3,y_d\big]\Big) \cong\PP(1,2,3,d)
$$
whose image is given by an equation of the~required form.
\end{proof}

\begin{remark}
\label{remark:canonical-embedding}
The embedding of the~surface $X$ described in Proposition~\ref{proposition:ind} is almost canonical.
It~only depends on the~choice of the~divisor class $A\in\Cl(X)$ such that $-K_X\sim \ind(X)A$, which~is uniquely defined modulo $\ind(X)$-torsion.
Thus, this embedding is $\Aut^0(X)$-equivariant.
\end{remark}

Using Proposition~\ref{proposition:ind}, we can describe many del Pezzo surfaces:

\begin{example}
\label{example:cubic-surface-2A2}
Suppose that $d=3$, $\ind(X)=3$, and $\Type(X)=2\type{A}_2$.
Using Proposition~\ref{proposition:ind}, we see that $X$ is a~hypersurface in $\PP(1,2,3,3)$ of degree $3$ that is given by
$$
y_3y_3'=y_2(y_2-y_1^2)(y_2-\lambda y_1^2),
$$
where $\lambda\in\Bbbk\setminus \{ 0, \, 1\}$.
\end{example}

\begin{example}
\label{example:Iskovskikh-surface}
Suppose that $d=4$, $\ind(X)=2$, and $\Type(X)=2\type{A}_1$.
Using Proposition~\ref{proposition:ind}, we see that $X$ is a~hypersurface in $\PP(1,1,2,2)$ of degree $4$ that is given by
$$
y_2y_2'=y_1y_1'(y_1'-y_1)(y_1'-\lambda y_1),
$$
where  $\lambda\in\Bbbk\setminus \{ 0, \, 1\}$. This surface is known as \emph{the~Iskovskikh surface} (see \cite{KunyavskijSkorobogatovTsfasman}).
\end{example}

Similarly, we can use Proposition~\ref{proposition:ind} to prove the~following result:

\begin{theorem}
\label{theorem:index>1}
Let $X$ be a Du Val del Pezzo surface, let $d:=K_X^2$.
Suppose that $d\geqslant 3$, $\ind(X)>1$, and the~surface $X$ is singular.
Then $X$ is a~hypersurface in a~weighted projective space $\PP$ such that one of the~following possibilities holds:
\newline\begin{center}\renewcommand{\arraystretch}{1.3}\rm
\begin{longtable}{|c|c|c|c|c|c|c|}
\hline
\quad \textnumero\quad&\quad$d$\quad\quad&\quad$\uprho$\quad\quad&\quad$\Type$\quad\quad&\quad$\ind$\quad\quad&\quad$\PP$\quad\quad&\quad\quad equation of $X$\quad\quad \quad \\
\hline
\ref{d=3:2A2}&\multirow3{*}{$3$}&$3$&$2\type{A}_2$&\multirow3{*}{$3$}&\multirow3{*}{$\PP(1,2,3,3)$ }&see Example~\ref{example:cubic-surface-2A2}\\
\cline{1-1}\cline{3-4}\cline{7-7}
\ref{d=3:2A2-A1}&&$2$&$2\type{A}_2\type{A}_1$&&&$y_3y_3'=y_2^2(y_2+y_1^2)$
\\
\cline{1-1}\cline{3-4}\cline{7-7}
\ref{d=3:3A2}&&$1$&$3\type{A}_2$&&&$y_3y_3'=y_2^3$
\\
\cline{1-1}\cline{3-4}\cline{7-7}
\ref{d=3:A5}&&$2$&$\type{A}_5$&&&$y_3^2=y_2^3+y_1^6+y_1y_2y_3'$
\\
\cline{1-1}\cline{3-4}\cline{7-7}
\ref{d=3:A5-A1}&&$1$&$\type{A}_5\type{A}_1$&&&$y_3^2=y_2^3+y_3'y_1y_2$
\\
\cline{1-1}\cline{3-4}\cline{7-7}
\ref{d=3:E6}&&$1$&$\type{E}_{6}$&&&$y_3^2=y_2^3+y_3'y_1^3$
\\
\hline
\ref{d=4:A3-A1}&\multirow3{*}{$4$}&$2$&$\type{A}_3\type{A}_1$&\multirow3{*}{$4$}&\multirow3{*}{$\PP(1,2,3,4)$ }&$y_3^2=y_1^6+y_2y_4$
\\
\cline{1-1}\cline{3-4}\cline{7-7}
\ref{d=4:A3-2A1}&&$1$&$\type{A}_32\type{A}_1$&&&$y_3^2=y_2y_4$
\\
\cline{1-1}\cline{3-4}\cline{7-7}
\ref{d=4:D5}&&$1$&$\type{D}_5$&&&$y_3^2=y_2^3+y_1^2y_4$
\\
\hline
\ref{d=4:2A1-8lines}&\multirow{7}{*}{$4$}&$4$&$2\type{A}_1$&\multirow{7}{*}{$2$}&\multirow3{*}{$\PP(1,1,2,2)$ }&see Example~\ref{example:Iskovskikh-surface}
\\
\cline{1-1}\cline{3-4}\cline{7-7}
\ref{d=4:3A1}&&$3$&$3\type{A}_1$&&&$y_2y_2'=y_1^2y_1'(y_1'+y_1)$
\\
\cline{1-1}\cline{3-4}\cline{7-7}
\ref{d=4:4A1}&&$2$&$4\type{A}_1$&&&$y_2y_2'=y_1^2y_1'^2$
\\
\cline{1-1}\cline{3-4}\cline{7-7}
\ref{d=4:A2-2A1}&&$2$&$\type{A}_22\type{A}_1$&&&$y_2y_2'=y_1^3y_1'$
\\
\cline{1-1}\cline{3-4}\cline{7-7}
\ref{d=4:A3-4lines}&&$3$&$\type{A}_3$&&&$y_2^2=y_2'y_1y_1'+y_1^4+y_1'^4$
\\
\cline{1-1}\cline{3-4}\cline{7-7}
\ref{d=4:D4}&&$2$&$\type{D}_4$&&&$y_2^2=y_2'y_1^2+y_1'^4$
\\
\hline
\ref{d=5:A4}&$5$&$1$&$\type{A}_4$&$5$&$\PP(1,2,3,5)$&$y_3^2+y_2^3+y_1y_5=0$
\\
\hline
\ref{d=6:A2-A1}&\multirow3{*}{$6$}&$1$&$\type{A}_2\type{A}_1$&$6$&$\PP(1,2,3)$&---
\\
\cline{1-1}\cline{3-7}
\ref{d=6:A1-3l}&&$3$&$\type{A}_1$&\multirow2{*}{$2$}&\multirow2{*}{$\PP(1,1,1,2)$ }&$y_1''y_2=y_1y_1'(y_1-y_1')$
\\
\cline{1-1}\cline{3-4}\cline{7-7}
\ref{d=6:2A1}&&$2$&$2\type{A}_1$&&&$y_1''y_2=y_1^2y_1'$
\\
\cline{1-1}\cline{3-7}
\ref{d=6:A2}&&$2$&$\type{A}_2$&$3$&$\PP(1,1,2,3)$&$y_1y_3=y_2^2-y_1'^4$
\\
\hline
\ref{d=8}&$8$&$1$&$\type{A}_1$&$4$&$\PP(1,1,2)$&---
\\
\hline
\end{longtable}
\end{center}
\end{theorem}

\begin{proof}
The required assertion follows from Proposition~\ref{proposition:ind}.
Let us show this in the~case $d=3$. Suppose that $d=3$ and $\ind(X)>1$.
Observe that $\ind(X)$ must divide $d$. Thus, we have $\ind(X)=3$.
By Proposition~\ref{proposition:ind}, $X$ is a~surface in $\PP(1,2,3,3)$ given by
$$
y_3^2+y_2^3+ \lambda_1 y_1^4y_2 +\lambda_2 y_1^6+\lambda_3x_3^2 +\lambda_4x_3y_1y_2+\lambda_5x_3y_1^3=0
$$
for some $\lambda_1$, $\lambda_2$, $\lambda_3$, $\lambda_4$ and $\lambda_5$ in $\Bbbk$.
If $\lambda_3\neq 0$, then completing the~square we reduce this equation to
the~defining equation of one the~surfaces~\ref{d=3:3A2}, \ref{d=3:2A2-A1} or~\ref{d=3:2A2},.
Thus, we may assume that $\lambda_3=0$.

If $\lambda_4\neq 0$, then we can use a~coordinate change
$x_3\mapsto \alpha x_3+\beta y_1^3$ and $y_2\mapsto \gamma y_2+\delta y_1^2$
for appropriate $\alpha$, $\beta$, $\gamma$ and $\delta$ in $\Bbbk$
to reduce our equation to
$$
y_3^2+y_2^3+ \lambda_2 y_1^6+x_3y_1y_2=0.
$$
If $\lambda_2=0$, this equation defines the~surface~\ref{d=3:A5-A1}.
On the~other hand, if $\lambda_2\ne 0$, we can scale the~coordinates to get $\lambda_2=1$,
so that we obtain the~defining equation of the~surface~\ref{d=3:A5}.

We may assume that $\lambda_3=\lambda_4=0$.
If $\lambda_5=0$, then $X$ has a~non-Du Val singularity at $(0:0:0:1)$, so that $\lambda_5\ne 0$.
Then we reduce our equation to the~defining equation of the~surface~\ref{d=3:E6}.
\end{proof}

\begin{remark}
Using Theorem~\ref{theorem:index>1}, we can easily obtain the~anticanonical embedding of the~surface~$X\hookrightarrow\PP^d$.
For instance, if $d=3$ and $\ind(X)=3$, the~map $\PP(1,2,3,3)\dasharrow\PP^3$ given by
$$
\big(y_1:y_2:y_3:y_3'\big)\longmapsto\big(y_1^3:y_1y_2:y_3:y_3'\big)
$$
defines an embedding $X\hookrightarrow \PP^3$, so that $X$ is a~cubic in $\PP^3$ given by
\begin{description}
\item[\ref{d=3:2A2}]
$x_0x_2x_3=x_1(x_1-x_0)(x_1-\lambda x_0)$, where $\lambda\in\Bbbk\setminus \{ 0,\,  1\}$;

\item[\ref{d=3:2A2-A1}]
$x_0x_2x_3=x_1^3+x_0x_1^2$;

\item[\ref{d=3:3A2}]
$x_0x_2x_3=x_1^3$;

\item[\ref{d=3:A5}]
$x_0x_2^2=x_1^3+x_0^3+x_0x_3x_1$;

\item[\ref{d=3:A5-A1}]
$x_0x_2^2=x_1^3 + x_0x_3x_1$;

\item[\ref{d=3:E6}]
$x_0x_2^2=x_1^3+ x_3x_0^2$.
\end{description}
\end{remark}

If $X$ is not weakly minimal, then $\ind(X)=1$.
In particular, if the~del Pezzo surface $X$ is smooth, then $\ind(X)=1$ unless $X\cong\PP^2$ or $X\cong\PP^1\times\PP^1$.
However, if $X$ is weakly minimal and~$d\geqslant 2$, we~cannot immediately conclude that $\ind(X)>1$.
Let us present two examples.

\begin{example}
\label{ex:A3quintic}
Let $X$ be a~quintic del Pezzo surface with $\uprho(X)=2$ admitting a conic bundle contraction $\psi_1\colon X\to \PP^1$.
It is easy to see from Lemma~\ref{extremal-contractions} that $\Type(X)=\type{A}_3$ and the~second extremal contraction is a birational $(1,4)$-contraction $\psi_2\colon X\to \PP^2$.
Then $X$ is weakly-minimal.
We have an $\Aut(X)$-equivariant morphism $\psi=(\psi_1,\psi_2)\colon X\to \PP^1\times \PP^2$ that is finite and birational onto its image,
which is given by
$$
\phi\big(v_0,v_1,u_0,u_1,u_2)=0,
$$
where $\phi$ is a bihomogeneous polynomial such that its degree with respect to $v_0,v_1$ equals $1$,
and its degree with respect to $u_0,u_1,u_2$ equals $2$, since $\psi_2$ is birational, $\psi_1$ is a conic bundle, $\uprho(X)=2$.
Let $P$ be the singular point of the surface $X$, and let $F$ be the fiber of $\psi_1$ that passes through $P$.
We may assume that $\psi(P)=(1: 0;\, 0:1:0)$.
Since $F$ is a multiple fiber of the conic bundle $\psi_1$, we may assume that $F$ is given by $u_2^2=0$.
Then
$$
\phi=u_2^2v_0+q(u_0,u_1,u_2)v_1,
$$
where $q$ is a quadratic form of rank $3$.
Changing coordinates, we may assume that $q=u_0^2+u_1u_2$, so~that $X$ is the~surface~\ref{d=5:A3}.
Let $\ind=\ind(X)$, and let $A$ be a Weil divisor on $X$ such that $-K_X\sim \ind A$.
Then $5=K_X^2=-\ind K_X\cdot A$ and $2=-K_X\cdot C=\ind A\cdot C$,
where $C$ is a general fiber of the~conic bundle~$\psi_1$.
Since $K_X\cdot A$ and $A\cdot C$ are integers, we have $\ind=1$.
\end{example}

\begin{example}
\label{ex:A4A1cubic}
Let $X$ be a~Du Val cubic surface in $\PP^3$ with $\Type(X)=\type{A}_4\type{A}_1$. Then $\uprho(X)=2$,
and~it~follows from \cite{Bruce1979} that $X$ is unique up to isomorphism and can be given by the~equation
$$
x_0x_2x_3+x^2_0 x_1+x^2_1x_3=0,
$$
so that $X$ is the~surfaces~\ref{d=3:A4-A1}.
Observe that $X$ contains exactly four lines:
$$
L_1=\{x_0=x_1=0\},\quad L_2=\{x_1=x_3=0\},\quad L_3=\{x_1=x_2=0\},\quad L_4=\{x_0=x_3=0\}
$$
and $-K_X\sim L_1+L_2+L_3\sim 2L_1+ L_4\sim 4L_1-L_2$. Then $\Cl(X)=\ZZ[L_1]\oplus \ZZ[L_2]$ and $\ind(X)=1$.
\end{example}

Thus, the~surfaces~\ref{d=3:A4-A1} or~\ref{d=5:A3} are weakly minimal and their Fano--Weil index is~$1$.
On the~other hand, we have the~following result:

\begin{theorem}
\label{theorem:Fano-Weil-index-big}
Let $X$ be a Du Val del Pezzo surface and let $d:=K_X^2$.
Suppose  $X$ is weakly minimal, $d\geqslant 3$, and $\ind(X)=1$.
Then $X$ is one of the~surfaces~\ref{d=3:A4-A1} or~\ref{d=5:A3}.
\end{theorem}

\begin{proof}
Observe that $\uprho(X)\geqslant 2$ by Lemma~\ref{lemma:rho=1=>index>1}.
First, let us consider the~case where $\uprho(X)=2$.
In~this case, the~Mori cone $\NE(X)$ is generated by two lines $L_1$ and $L_2$ such that $L_1\cap L_2\neq \varnothing$.
Without loss of generality, we may assume that $L_1^2\geqslant L_2^2$.
Then
$$
-K_X\simQ \alpha_1L_1+\alpha_2L_2
$$
for some $\alpha_1\in\QQ_{>0}$ and $\alpha_2\in\QQ_{>0}$.
Since $-K_X\cdot L_1=-K_X\cdot L_2=1$ and $K_X^2=d$, we get
\begin{equation}
\label{eq:rho=2}
\left\{\aligned
&\alpha_1+\alpha_2=d,\\
&\alpha_1 L_1^2 +\alpha_2 L_1\cdot L_2=1,\\
&\alpha_1 L_1\cdot L_2 +\alpha_2 L_2^2=1.\\
\endaligned
\right.
\end{equation}
Let $\psi_1\colon X\to Y_1$ and $\psi_2\colon X\to Y_2$ be the~contractions of the~extremal rays that are generated by the~lines $L_1$ and $L_2$, respectively.

Assume that $\psi_1$ is a conic bundle. Since $X$ is weakly minimal, by Lemma~\ref{lemma:conic-bundles}\ref{conic:bundle:3} we have $d\geqslant 4$.
In particular, we see that the~anticanonical model of the surface $X$ is an intersection of quadrics.
Let~$C_1$ be a general fiber of $\psi_1$.
Then $C_1\cdot L_2=1$ and $C_1\sim 2L_1$. Hence, $L_1\cdot L_2=\frac{1}{2}$ and $L_1^2=0$.
Then \eqref{eq:rho=2} gives $\alpha_2=2$, $\alpha_1=d-2$ and $L_2^2=1-\frac{d}{4}$.
To proceed, we may assume that $d\leqslant 6$.
If~$d=5$, then $L_2^2=-\frac{1}{4}$ and $\psi_2$ is an $(1,4)$-contraction by Lemma~\ref{extremal-contractions}, so that $X$ is the~surface~\ref{d=5:A3}.
If~$d=6$,~then $\alpha_1=4$, $L_2^2=-\frac{1}{2}$ and $\psi_2$ is an $(1,2)$-contraction. Then $4L_1+2L_2$ is a~Cartier divisor, so that
$$
-K_X\sim 4L_1+2L_2=2(2L_1+L_2),
$$
which is impossible, since $\ind(X)=1$. Finally, if $d=4$, then $L_2^2=0$ and $\psi_2$ is also a conic bundle.
Since $2L_1+2L_2$ is Cartier, we have $-K_X\sim 2L_1+2L_2$ and so $\ind(X)>1$, a contradiction.

Thus, we may assume that both $\psi_1$ and $\psi_2$ are birational.

Each line $L_1$ and $L_2$ contains exactly one singular point of the~surface $X$ by Lemma~\ref{extremal-contractions}.
Let~$P_1$ be the~singular point contained in $L_1$, and let $P_2$ be the~singular point contained in $L_2$.
By Lemma~\ref{extremal-contractions}, the~points $P_1$ and $P_2$ are singular points of types $\type{A}_{n_1}$ and $\type{A}_{n_2}$
for some positive integers $n_1$ and $n_2$.
Then
$$
-\frac1{n_1+1}=L_1^2\geqslant L_2^2=-\frac1{n_2+1}
$$
by Lemma~\ref{extremal-contractions}, so that $n_1\geqslant n_2\geqslant 1$.

Suppose that $P_1\ne P_2$. Then $L_1\cap L_2$ is a~smooth point of the~surface $X$, so that $L_1\cdot L_2=1$.
Then \eqref{eq:rho=2} gives
$$
\left\{\aligned
&\alpha_1+\alpha_2=d,\\
&-\alpha_1 +\alpha_2 (n_1+1)=n_1+1,\\
&\alpha_1 (n_2+1)-\alpha_2=n_2+1.\\
\endaligned
\right.
$$
Note also that $d+n_1+n_2\leqslant 8$, since $\uprho(\widetilde X)=10-d$.
Eliminating $\alpha_1$ and $\alpha_2$, we get
$$
d(n_1n_2 +n_2 +n_1) = 2n_1n_2 +3n_1+3n_2 +4
$$
This give us the~following solutions:
\begin{itemize}
\item
$d=4$, $n_1=n_2=1$, $\alpha_1=\alpha_2=2$, $-K_X\simQ 2(L_1+L_2)$,
\item
$d=3$, $n_1=n_2=2$, $\alpha_1=\alpha_2=3/2$, $-K_X\simQ \frac 32(L_1+L_2)$,
\item
$d=3$, $n_1=4$, $n_2=1$, $\Type(X)=\type{A}_4\type{A}_1$, $\alpha_1=\frac53$, $\alpha_2=\frac{4}3$, $-K_X\simQ \frac13(5L_1+4L_2)$.
\end{itemize}
If $d=4$, then $2(L_1+L_2)$ is a~Cartier divisor, which gives
$$
-K_X\sim 2(L_1+L_2),
$$
which is impossible.
If $d=3$ and $n_1=n_2=2$, then $X$ is a~cubic surface in $\PP^3$,
so that $X$ contains a~line $L$ such that $L$ passes through $P_1$ and $P_2$ and
$$
-K_X\sim L_1+L_2+L
$$
which gives $-K_X\simQ 3L$, because $-K_X\simQ\frac{3}{2}(L_1+L_2)$.
In this case, the~divisor $3L$ is Cartier, so~that $-K_X\sim 3L$, which contradicts $\ind(X)=1$.
Thus, we conclude that $d=3$, $n_1=4$ and~$n_2=1$. Then $X$ is the~surface~\ref{d=3:A4-A1}.
Hence, to proceed, we may assume that $P_1=P_2$.

Let $n=n_1=n_2$. Then $L_1^2=L_2^2=-\frac1{n+1}$. Moreover, we have $L_1\cdot L_2=\frac{k}{n+1}$ for some $k\in \ZZ_{>0}$.
Then \eqref{eq:rho=2} gives $\alpha_1=\alpha_2=\frac{d}2$ and $2(n+1)=d(k-1)$, so that
$$
-K_X\simQ \textstyle{ \frac{d}2} \big(L_1+L_2\big).
$$
Note also that $d+n\leqslant 8$, since $\uprho(\widetilde X)=10-d$.
Then $d\ne 5$ and $d\ne 7$, because $2(n+1)=d(k-1)$.
Likewise, if $d=6$, then $n=2$, so that $-K_X\simQ 3(L_1+L_2)$ and $\Cl(X)$ is torsion free by Lemma~\ref{lemma:Cl},
which gives $-K_X\sim 3(L_1+L_2)$, which contradicts $\ind(X)=1$.
Hence, either $d=3$ or $d=4$.

Suppose that $d=4$. Then $n=3$, because $2(n+1)=d(k-1)$ and $d+n\leqslant 8$.
Since $\uprho(X)=2$, we see that $X$ has a~singular point of type $\type{A}_1$.
On the~other hand, we have $-K_X\simQ 2(L_1+L_2)$.
Since $\ind(X)=1$, we have $-K_X\not\sim 2(L_1+L_2)$, so that $K_X+2(L_1+L_2)$ is a~non-trivial torsion element in $\Cl(X)$.
Now, applying Lemma~\ref{lemma:Cl}, we obtain a~double cover $\pi\colon Y\to X$ that is \'etale outside of the~point $L_1\cap L_2$.
Then $Y$ is a~del Pezzo surface of degree $8$ such that it contains two singular points of type $\type{A}_1$, which is absurd.
This shows that $d\ne 4$.

Therefore, we see that $d=3$. Since $2(n+1)=d(k-1)$ and $d+n\leqslant 8$, we have $n=2$ or $n=5$.
Since $X$ is a~cubic surface in $\PP^3$, it contains a~line $L$ such that
$$
-K_X\sim L_1+L_2+L,
$$
so that $L\simQ\frac12(L_1+L_2)$, because $-K_X\simQ\frac32(L_1+L_2)$. This gives $-K_X\simQ 3L$. But $-K_X\not\sim 3L$.
In~particular, we have $n\ne 2$, because $3L$ is a~Cartier divisor if $n=2$.
We conclude that $n=5$.
Now, using Lemma~\ref{lemma:Cl}, we conclude that there is a~finite Galois cover $\pi\colon Y\to X$ of degree $r\geqslant 2$,
which is \'etale outside of the~point $L_1\cap L_2$. Here, $r$ be the~order of the~torsion divisor $K_X+3L$.
By~construction, the~surface $Y$ is a~del Pezzo surface of degree $rd$, so that either $r=2$ or $r=3$.
If $r=3$, then $Y\cong\PP^2$, which is impossible, since $\uprho(Y)\geqslant \uprho(X)=2$.
Thus, we have $r=2$. Then
$$
-K_X\sim 3L-(K_X+3L)\sim 3L-(K_X+3L)+4(K_X+3L)\sim 3L+3(K_X+3L)\sim 3(K_X+4L),
$$
which is impossible, since $\ind(X)=1$. The obtained contradiction shows that $\uprho(X)\ne 2$.

We see that $\uprho(X)\geqslant 3$ and $X$ is singular. Then $d\leqslant 6$.

Let $\psi\colon X\to Y$ be an extremal Mori contraction.
Since $\uprho(X)\geqslant 3$, the~morphism $\psi$ is birational,
so that $\psi$ is a $(1,m)$-contraction  of a line $L\subset X$ by Lemma~\ref{extremal-contractions}.
Then $Y$ is a weakly minimal Du Val del Pezzo surface such that $K_Y^2=d+m$ with $\uprho(Y)=\uprho(X)-1\geqslant  2$ (see Remark~\ref{remark:min}).
In~particular, we have $d+m\ne 7$, because a Du Val del Pezzo surface of degree $7$ is not weakly minimal (see Example~\ref{example:dP7-A1}).
Note that $m\geqslant 2$, since $X$ is weakly minimal.

Consider the~case $d=6$.
Using the~Noether formula, we see that $\uprho(X)=3$ and $\Type(X)=\type{A}_1$.
Then $m=2$, $K_Y^2=8$, $\uprho(Y)=2$, so that $Y\cong \PP^1\times\PP^1$.
Then $K_X\sim\psi^* K_Y+2L$ is divisible by~$2$, which is a contradiction.
Thus, we have $d\ne 6$.

Consider the~case $d=5$.
Since $d+m\ne 7$, we have $m>2$.
Then $\uprho(X)=3$ and $\Type(X)=\type{A}_2$ by the~Noether formula,
so that $X$ is not weakly minimal by \cite[Proposition~8.5]{Coray1988}. This contradicts our assumption.

Consider the~case $d=4$.
Then $m\ne 3$,  since $d+m\ne 7$.
If $m>3$, it follows from the~Noether formula that $m=4$, $\uprho(X)=3$ and $\Type(X)=\type{A}_3$,
so that $Y\cong\mathbb{P}^1\times\mathbb{P}^1$, which gives
$$
-K_X\sim 2\big(\widetilde{F}_1+\widetilde{F}_2\big),
$$
where $\widetilde{F}_1$ and $\widetilde{F}_2$ are proper transforms on $X$ of the~curves in $Y$ of bi-degree $(1,0)$ and $(0,1)$
that contains $\psi(L)$, respectively.
This contradicts our assumption $\ind(X)=1$. So, we see that all extremal contractions on $X$ are birational $(1,2)$-contractions.
Then $\ind(Y)>1$, since we already dealt with sextic del Pezzo surfaces.
By Theorem~\ref{theorem:index>1}, we see that  $Y$ is one of the~surfaces \ref{d=6:A1-3l}, \ref{d=6:2A1}, \ref{d=6:A2}.
If $Y$ is the~surface~\ref{d=6:A2}, then it has a birational $(1,3)$-contraction,
so that $X$ also has a birational $(1,3)$-contraction by Corollary~\ref{corollary:extremal-contractionsDP},
which is a contradiction.
Then $Y$ is one of the~surfaces \ref{d=6:A1-3l} or \ref{d=6:2A1}.
By Lemma~\ref{lemma:Cl}\ref{lemma:Cl:gens}, we~have
$$
-K_Y\sim 2\sum_{i=1}^{s} a_i M_i
$$
for some lines $M_1,\ldots,M_s$ on the~surface $Y$ and some integers $a_1,\ldots,a_s$.
Since $\psi(L)$ is a smooth point of $Y$ that does not lie on a line by Corollary~\ref{corollary:extremal-contractionsDP},
we obtain
$$
-K_X\sim 2\sum_{i=1}^s a_i \widetilde{M}_i-2L\sim2\Big(\sum_{i=1}^s a_i \widetilde{M}_i-L\Big),
$$
where $\widetilde{M}_i$ is a proper transform on $X$ of the~line $M_i$. This contradicts our assumption $\ind(X)=1$.

Finally, we consider the~case $d=3$.
Then $m\ne 4$,  since $d+m\ne 7$.
Moreover, since $X$ is weakly minimal, there exists no dominant morphisms from $X$ to a curve by Lemma~\ref{lemma:conic-bundles}\ref{conic:bundle:3}, and the~same holds for $Y$.
Using this, we conclude that $m\ne 5$.
Thus, we have the~following possibilities:
\begin{itemize}
\item either $K_Y^2=5$ and $m=2$,
\item or $K_Y^2=6$ and $m=3$.
\end{itemize}
Moreover, if $K_Y^2=5$, then $Y$ is not the~surface \ref{d=5:A3}, because del Pezzo surface \ref{d=5:A3} admits a~dominant morphism to $\PP^1$ (see Example~\ref{ex:A3quintic}).
Therefore, we conclude that $\ind(Y)>1$, because we already dealt with weakly minimal Du Val del Pezzo surfaces of degree $5$ and $6$.

Now, using Theorem~\ref{theorem:index>1}, we see that $K_Y^2\ne 5$, because $\uprho(Y)>1$.
Therefore, we have $K_Y^2=6$. Then $Y$  is the~surface \ref{d=6:A2}, \ref{d=6:2A1} or \ref{d=6:A1-3l} again by Theorem~\ref{theorem:index>1}.
If $Y$  is the~surface \ref{d=6:2A1}, then
$$
\frac{y_1}{y_1''}=\frac{y_1'^2}{y_2}
$$
on the surface $Y$, so that the~map $Y\dashrightarrow \PP^1$ given by
$$
(y_1:y_1':y_1'':y_2)\longmapsto(y_1:y_1'')=(y_1'^2:y_2)
$$
is a morphism, which is a contradiction. Similarly, we obtain a contradiction when $Y$  is the~surface \ref{d=6:A1-3l},
because the~map $Y \dashrightarrow \PP^1$ given by
$$
(y_1:y_1':y_1'':y_2)\longmapsto(y_1:y_1')=(y_1''(y_1'+y_1''):y_2)
$$
is a morphism in this case. Thus, we see that $Y$ is the~surface \ref{d=6:A2}.
Then $X$ is a cubic surface such that $\Type(X)=2\type{A}_2$.
Now, using \cite{Bruce1979}, we conclude that $X$ is one of the~surfaces~\ref{d=3:2A2},
so that $\ind(X)=3$ by Theorem~\ref{theorem:index>1},
which contradicts our assumption.
\end{proof}

\section{The proof of Main Theorem: higher degree cases}
\label{section:dP-large-degree1}
Let $X$ be a Du Val del Pezzo surface of degree $d$ whose automorphism group $\mathrm{Aut}(X)$ is infinite.
If $d\geqslant 8$, then $X$ is either
$\PP^2$, $\PP^1\times\PP^1$,  $\FF_1$ or $\PP(1,1,2)$.
In each of this case, the~corresponding automorphism group is well-known and listed in Big Table
(cases
\ref{d=9:P2}, \ref{d=8:P1-P1}, \ref{d=8:F1}, \ref{d=8} respectively).

If $d\leqslant 7$, we have an $\Aut^0(X)$-equivariant diagram \eqref{equation:min-resolution}.
If $d=7$, then $X$ is one of the~del Pezzzo surfaces \ref{d=7} and \ref{d=7:smooth},
and the~morphism $\varphi$ in \eqref{equation:min-resolution} is a blow up of two (possibly infinitely near) points.
From this we obtain the~following

\begin{lemma}
\label{lemma:DP7}
Let $X$ be a Du Val del Pezzo surface of degree $7$,
and let $U$ be the~complement in the~surface $X$ to the~union of all lines.
Then Main Theorem holds for $X$, the~subset $U$ is the~open orbit of the~group $\Aut^0(X)$, and $U\cong\Aff^2$.
\end{lemma}

All del Pezzo surfaces of degree $6$ have infinite automorphism groups, so that all of them appear in our Big Table.
These are the~del Pezzo surfaces~\ref{d=6:A2-A1}, \ref{d=6:A2}, \ref{d=6:2A1}, \ref{d=6:A1-3l}, \ref{d=6:A1-4l} and~\ref{d=6:smooth}.
Going through these six cases one by one, we obtain

\begin{lemma}
\label{corollary:d-6}
Let $X$ be a Du Val del Pezzo surface of degree $6$,
and let $U$ be the~complement in the~surface $X$ to the~union of all lines.
Then Main Theorem holds for $X$, the~subset $U$ is the~open orbit of the~group $\Aut^0(X)$, and $U\cong \Aff^2$ in the~cases~\ref{d=6:A2-A1}, \ref{d=6:A2}, \ref{d=6:2A1}, and \ref{d=6:A1-3l}.
\end{lemma}

\begin{proof}
If $X$ is the~surface~\ref{d=6:smooth}, then $\varphi$ in \eqref{equation:min-resolution} is a blow up of three distinct non-collinear points,
so that $X$ is toric and $\Aut^0(X)\cong \Gm^2$.

Likewise, if $X$ is the~surface~\ref{d=6:A1-3l}, then $\varphi$ is the~blow up of three distinct collinear points in~$\PP^2$.
Using this, it is not hard to see that $\Aut^0(X)\cong \Ga^2\rtimes\Gm$ in this case.

For infinitely near points we use the~notation of \cite{Dolgachev}.

If $X$ is the~surface~\ref{d=6:A2-A1}, then the~morphism $\varphi$ is the~blow up of three infinitely near collinear points $P_1\prec P_2\prec P_3$ in the~plane $\PP^2$ in the~notations of \cite{Dolgachev},
which implies that $\Aut^0(X)\cong \BB_3$.

Similarly, if $X$ is the~surface~\ref{d=6:A2}, then $\varphi$ is the~blow up of three infinitely near non-collinear points $P_1\prec P_2\prec P_3$,
which implies that $\Aut^0(X)\cong \UU_3\rtimes \Gm$.

If $X$ is the~surface~\ref{d=6:2A1}, then $\varphi$ is the~blow up of three collinear points $P_1$, $P_2$ and $P_3$ such that the~points $P_1$ and $P_2$ are distinct and $P_3\succ P_1$,
which implies that $\Aut^0(X)\cong \BB_2\times\BB_2$.

Finally, if $X$ is the~surface~\ref{d=6:A1-4l}, then $\varphi$ is the~blow up of three non-collinear points $P_1$, $P_2$ and~$P_3$,
so that $P_1$ and $P_2$ are distinct, but $P_3\succ P_1$.
Hence, in this case, we have $\Aut^0(X)\cong \BB_2\times \Gm$.

The last assertions follow from Corollary~\ref{corollary:Ga} in the~cases~\ref{d=6:A2-A1}, \ref{d=6:A2}, \ref{d=6:2A1}, and \ref{d=6:A1-3l},
and it follows from  Lemma~\ref{lemma:DP7} in~the~cases \ref{d=6:A1-4l} and~\ref{d=6:smooth}.
\end{proof}

Similarly, all singular del Pezzo surfaces of degree $5$ also have infinite automorphism groups.
These are the~surfaces \ref{d=5:A4}, \ref{d=5:A3}, \ref{d=5:A2-A1}, \ref{d=5:A2},  \ref{d=5:2A1} and \ref{d=5:A1} in Big Table.

\begin{lemma}
\label{lemma:d=5}
Let $X$ be a Du Val del Pezzo surface of degree $5$,
and let $U$ be the~complement in the~surface $X$ to the~union of all lines.
Then Main Theorem holds for $X$, the~subset $U$ is the~open orbit of the~group $\Aut^0(X)$
in the~cases~\ref{d=5:A4}, \ref{d=5:A3}, \ref{d=5:A2-A1}, \ref{d=5:A2}, \ref{d=5:2A1},
and $U\cong \Aff^2$ in the~cases~\ref{d=5:A4}, \ref{d=5:A3}.
\end{lemma}

\begin{proof}
If $X$ is weakly minimal and $\ind(X)>1$, then $X$ is the~surface \ref{d=5:A4} by Theorems~\ref{theorem:Fano-Weil-index-big} and~\ref{theorem:index>1}.
Then the~group $\Aut^0(X)$ consists of the~transformations that send the point $(y_1:y_2:y_3:y_5)$ to
$$\hspace*{-0.4cm}
\Big(y_1:t^2y_2+ay_1^2:t^3y_3+by_1^3+cy_1y_2:t^6y_5-(a^3+b^2)y_1^5-(3a^2t^2+2bc)y_1^3y_2-2bt^3y_1^2y_3-(3at^4+c^2)y_1y_2^2-2ct^3y_2y_3\Big),
$$
where $t\in\Bbbk^\ast$ and $a,b,c\in\Bbbk$.
This gives $\Aut^0(X)\cong\UU^3\rtimes\Gm$ (cf. Corollary~\ref{corollary:connected-group-rational-surface}, see Corollary~\ref{Cor:Cl=Z:Aut}).

If $X$ is weakly minimal and $\ind(X)=1$, then $X$ is the~del Pezzo surface \ref{d=5:A3}  by Theorem~\ref{theorem:Fano-Weil-index-big},
and its $\mathrm{Aut}^0(X)$-equivariant embedding $X\hookrightarrow\PP^1\times \PP^2$ is described in Example~\ref{ex:A3quintic}.
In~this case, the~group $\Aut(X)$ contains a~two-dimensional unipotent subgroup
\[
(v_0:v_1;\, u_0:u_1:u_2)\longmapsto\big(v_0-(a_1^2+a_2)v_1:v_1;\, u_0+a_1u_2:u_1-2a_1u_0+a_2 u_2:u_2\big)
\]
and a one-dimensional torus
\[
(v_0:v_1;\, u_0:u_1:u_2)\longmapsto (v_0:t^{-2} v_1;\, t u_0:t^2 u_1: u_2),
\]
where $a_1,a_2\in\Bbbk$ and $t\in\Bbbk^\ast$. This implies that $\Aut(X)\cong\Ga^2\rtimes\Gm$ as required.

We may assume that $X$ is not weakly minimal.
Then there is a birational morphism $\psi\colon X\to Y$
such that $Y$ is one of the~surfaces \ref{d=6:A2-A1}, \ref{d=6:A2}, \ref{d=6:2A1}, \ref{d=6:A1-3l},
and $\psi$ is a blow up of a smooth point $P\in Y$.
By~Corollary~\ref{corollary:extremal-contractionsDP} and Lemma~\ref{corollary:d-6},
the~point $P$ is contained in the~open orbit of the~group $\mathrm{Aut}^0(Y)$.
Since $\Aut^0(X)$ is the~connected component of the~stabilizer in $\Aut^0(Y)$ of the~point $P$,
this gives the~required description of the~group $\mathrm{Aut}^0(X)$ in Big Table.

The last assertions follow from Corollary~\ref{corollary:Ga} in the~cases~~\ref{d=5:A4}, \ref{d=5:A3},
and it follows from Lemma~\ref{corollary:d-6} in the~cases \ref{d=5:A2-A1}, \ref{d=5:A2}, \ref{d=5:2A1}.
\end{proof}

Let us conclude this section by proving Main Theorem for Du Val del Pezzo surfaces of degree~$4$.

\begin{proposition}
\label{proposition:d=4-ind-aut}
Main Theorem holds for Du Val del Pezzo surfaces of degree $4$.
\end{proposition}

\begin{proof}
Let $X$ be Du Val del Pezzo surface of degree $4$ such that the~group $\Aut(X)$ is infinite.
Suppose that $X$ is not weakly minimal.
Then there exists a birational morphism $\psi\colon X\to Y$
such that $Y$ is a singular quintic Du Val del Pezzo surface,
and $\psi$ is a blow up of a smooth point~$P\in Y$, which is not contained in a line by Corollary~\ref{corollary:extremal-contractionsDP}.
Observe that the~group $\Aut^0(X)$ is the~connected component of the~stabilizer in $\Aut^0(Y)$ of the~point~$P$.
Using Lemma~\ref{lemma:d=5}, we see~that $Y$ must be one of the~surfaces \ref{d=5:A4}, \ref{d=5:A3}, \ref{d=5:A2-A1},
and $P$ must be contained in the~open orbit of the~group $\mathrm{Aut}^0(Y)$.
This implies that $X$ is one of the~surfaces~\ref{d=4:A4}, \ref{d=4:A3-5lines}, \ref{d=4:A2-A1}, respectively.
Now, it is not hard to check that $\Aut^0(X)\cong \Gm$ in the~cases~\ref{d=4:A3-5lines} and~\ref{d=4:A2-A1},
and $\Aut^0(X)\cong \BB_2$ in the~case~\ref{d=4:A4}.

Hence, we may assume that the~surface $X$ is weakly minimal.
Then $\ind(X)>1$ \mbox{by~Theorem~\ref{theorem:Fano-Weil-index-big}}.
Using Theorem~\ref{theorem:index>1}, we see that $X$ is one of the~surfaces~\ref{d=4:D5}, \ref{d=4:A3-2A1}, \ref{d=4:D4}, \ref{d=4:A3-A1}, \ref{d=4:A2-2A1}, \ref{d=4:4A1}, \ref{d=4:A3-4lines},
\ref{d=4:3A1}, \ref{d=4:2A1-8lines}.
Let us show that $\Aut(X)$ is infinite, and $\Aut^0(X)$ is described in Big Table.

Let $X$ be the~surface~\ref{d=4:D5}.
Then $X$ is embedded into $\PP(1,2,3,4)$ as a hypersurface
that is given by the~equation $y_3^2=y_2^3+y_1^2y_4$.
This embedding is $\Aut^0(X)$-equivariant by Remark~\ref{remark:canonical-embedding}.
Observe that $\Aut^0(X)$ contains transformations
$$
(y_1:y_2:y_3:y_4)\longmapsto \big(y_1:y_2+ay_1^2: y_3+by_1^3:y_4-(a^3-b^2)y_1^4-3a^2y_1^2y_2+2by_1y_3-3ay_2^2\big),
$$
where $a\in\Bbbk$ and $b\in\Bbbk$. These transformations generates a subgroup in $\Aut^0(X)$ isomorphic to~$\Ga^2$.
Moreover, the~surface $X$ also admits an action of a~one-dimensional torus which acts diagonally:
$$
(y_1:y_2:y_3:y_4)\longmapsto (t^2 y_1: t^2y_2: t^3y_3: t^2y_4),
$$
where $t\in\Gm$. The described transformations generate a subgroup that is isomorphic to $\Ga^2\rtimes\Gm$.
Since $X$ has a singularity of type $\type{D}_5$, it is not toric, so that $\Aut^0(X)\cong \Ga^2\rtimes\Gm$
(cf. Corollary~\ref{Cor:Cl=Z:Aut}).

Now we suppose that $X$ is the~surface~\ref{d=4:A3-2A1}. Then $\Aut(X)$ contains transformations
$$
\gamma(a): (y_1:y_2:y_3:y_4)\longmapsto\big(y_1:y_2:y_2+ay_1y_2:y_4+2ay_1x_2+a^2y_1^2y_2\big)
$$
for every $a\in\Bbbk$. These transformations generates a proper subgroup in $\Aut^0(X)$ isomorphic to $\Ga$.
Moreover, the~surface $X$ also contains transformations
$$
\delta(t_1,t_2): (y_1:y_2:y_3:y_4)\longmapsto \big(y_1:t_1t_2^2 y_2:t_1t_2y_3:t_1y_4\big)
$$
for every $t_1\in\Bbbk^\ast$ and $t_2\in\Bbbk^\ast$.
They generates a subgroup isomorphic to $\Gm^2$.
Observe that
$$
\gamma(a)\circ\delta(t_1,t_2)=\delta(t_1,t_2)\circ\gamma_{at_2}.
$$
Therefore, all described transformations generate a subgroup in $\Aut^0(X)$ isomorphic to $\BB_2\times \Gm$.
Then $\Aut^0(X)\cong \BB_2\times \Gm$, because $X$ does not admit an effective $\Ga^2$-action by Corollary~\ref{corollary:Ga}.

Now we suppose that~\ref{d=4:A3-A1}.
Then the~group $\Aut(X)$ contains transformations
$$
(y_1:y_2:y_3:y_4)\longmapsto\big(ty_1:t^3y_2:t^3y_3+at^3y_1y_2: t^3y_4+2at^3y_1y3+a^2t^3y_1^2y_2\big)
$$
for any $a\in\Bbbk$ and $t\in\Bbbk^\ast$. These transformations generate a subgroup in $\Aut^0(X)$ isomorphic~to~$\BB_2$.
Since all three lines on $X$ pass through one point, $X$ is not toric.
Hence, $\rk\Aut^0(X)=1$.
The~surface $X$ does not admit an~effective~$\Ga^2$-action by Corollary~\ref{corollary:Ga},
so that $\dim \Aut^0(X)=2$, which implies $\Aut^0(X)\cong \BB_2$.

Let $X$ be one of the~surfaces~\ref{d=4:A2-2A1} or~\ref{d=4:4A1}. Then $\numl(X)=4$.
Let $L_1$, $L_2$, $L_3$, $L_4$ be the~lines in~$X$.
Recall that $X$ is an intersection of two quadrics in $\PP^4$. We have
$$
L_1+L_2+L_3+L_4\sim -K_X,
$$
so that $L_1+L_2+L_3+L_4$ is cut out by a hyperplane $H\subset\PP^4$.
On~the~other hand, this curve form a~combinatorial cycle.
Thus, if $X$ admits an effective \mbox{$\Ga$-action}, then this action is trivial on
each line among $L_1$, $L_2$, $L_3$ and $L_4$, so that it is trivial on $H$,
which implies that the~closure of any one-dimensional $\Ga$-orbit is a~line.
The latter is impossible, since $X$ contains finitely many~lines.
Therefore, we conclude that the~surface $X$ does not admit an effective action of the~group $\Ga$.
On the~other hand, the~equations of $X$ are binomial. This implies that $X$
admits a diagonal action of two-dimensional torus.
Hence, $\Aut^0(X)\cong \Gm^2$.

Let $X$ be one of the~surfaces~\ref{d=4:3A1} or \ref{d=4:2A1-8lines}.
Then $X$ contains two lines $L_1$ and $L_2$ such that the~intersection $L_1\cap L_2$ is a smooth point of $X$,
and there exist an $\Aut^0(X)$-equivariant diagram
$$
\xymatrix@R=0.8em{
&\widehat{X}\ar@{->}[rd]^{\beta}\ar@{->}[dl]_{\alpha}&\\%
X&&Y}
$$
where $\alpha$ is a blow up of the~point $L_1\cap L_2$, and $\beta$ is the~birational contraction of the~proper transforms of the~lines $L_1$ and $L_2$.
Then we have the~following possibilities:
\begin{itemize}
\item
if $X$ is the~surface~\ref{d=4:3A1}, then $Y$ is a cubic surface such that $\Type(Y)=\type{A}_4\type{A}_1$;
\item
if $X$ is the~surface~\ref{d=4:2A1-8lines}, then $Y$ is a cubic surface such that $\Type(Y)=2\type{A}_2$.
\end{itemize}
Hence, it follows from Corollary~\ref{corollary:d-3}, that $Y$ does not admit an effective action of the~group $\Ga$.
Since $X$ contains three lines passing through one point, it is not toric.
One the~other hand, it is easy to see that $X$ admits an effective diagonal action of a~one-dimensional torus.

To complete the~proof, we may assume that $X$ is one of the~surfaces~\ref{d=4:D4} or~\ref{d=4:A3-4lines}.
By~Lemma~\ref{lemma:conic-bundles}\ref{conic:bundle:4},
there is a double cover $\pi\colon X\to\PP^1\times\PP^1$ branched over a curve $B$ of degree~$(2,2)$.
By construction, this double cover is $\Aut^0(X)$-equivariant,
and the~curve $B$ is $\Aut^0(X)$-invariant.
Therefore, there exists an exact sequence of groups
$$
1\longrightarrow\mumu_2\longrightarrow\Aut(X)\longrightarrow\Aut\big(\PP^1\times\PP^1,B\big).
$$
If~$X$ is the~surface~\ref{d=4:D4}, then $B$ is a union of irreducible smooth curves of degrees $(1,1)$, $(1,0)$,~$(0,1)$, which intersect in one point,
which implies that
$$
\Aut\big(\PP^1\times\PP^1,B\big)\cong\BB_2\rtimes\mumu_2.
$$
This can be shown by taking linear projection $\PP^1\times\PP^1\dasharrow\PP^2$ from the~singular point $\Sing(B)$,
where we consider $\PP^1\times\PP^1$ as a quadric in $\PP^3$.
Thus, in this case, we have $\Aut^0(X)\cong \Ga \rtimes_{(2)} \Gm$,
since $\Aut(X)$ contains a subgroup isomorphic to $\Ga\rtimes_{(2)}\Gm$ generated by transformations
$$
(y_1:y_1':y_2:y_2')\longmapsto\big(y_1:ty_1':t^2y_2+at^2y_1^2:t^4y_2'+2at^4y_2+a^2t^4y_1^2\big),
$$
where $a\in\Bbbk$ and $t\in\Bbbk^\ast$. Similarly, if $X$ is the~surface~\ref{d=4:A3-4lines}, then
$$
\Aut\big(\PP^1\times\PP^1,B\big)\cong\Ga\rtimes\mumu_2,
$$
because $B$ is a union of two irreducible smooth curves of degree $(1,1)$, which intersect in one point.
In this case, we have $\Aut^0(X)\cong\Ga$, since the~$\Ga$-action lifts from $\PP^1\times\PP^1$ to the~surface $X$.
\end{proof}

\begin{corollary}
\label{cor:high-degree}
Main Theorem holds del Pezzo surfaces of degree $\geqslant 4$.
\end{corollary}

\section{Del Pezzo surfaces of degree $1$}
\label{section:deg-1}
In this section, we prove Main Theorem for del Pezzo surfaces of degree $1$.
\begin{proposition}
\label{prop:d-1}
Main Theorem holds for del Pezzo surfaces of degree $1$.
\end{proposition}

% Now, let us prove Main Theorem for del Pezzo surfaces of degree $1$.
We start with

\begin{lemma}
\label{lemma:dP1-Ga}
Let $X$ be a~Du Val del Pezzo surface of degree $1$.
Then $X$ does not admit effective actions of the~group $\Ga$.
\end{lemma}

\begin{proof}
Suppose that $X$ admit an effective $\Ga$-action.
Let $\Phi\colon X\dasharrow\PP^1$ be the~anticanonical map.
It is $\Aut(X)$-equivariant,
and all its~fibers are reduced irreducible curves of arithmetic genus $1$.
Since $\Ga$ cannot effectively act on a smooth elliptic curve, we conclude that $\Ga$ acts non-trivially on the~base of $\Phi$.
Thus, there is exactly one $\Ga$-invariant fiber, say $C$.
Any fiber of $\Phi$ different from $C$ is a smooth elliptic curve.
Thus we have
\[
\uprho(X)+2=\chi(X)=\chi(C)-1,
\]
so that $\chi(C)=\uprho(X)+3\geqslant 4$. But $\chi(C)\leqslant 2$, because $C$ is a curve of arithmetic genus $1$.
\end{proof}

Our next step in proving Main Theorem for del Pezzo surfaces of degree $1$ is the~following

\begin{lemma}
\label{lemma:degree-1-Gm-Aut}
If $X$ is one of the~surfaces~\ref{d=1:E8}, \ref{d=1:E7-A1}, \ref{d=1:E6-A2}, \ref{d=1:2D4}, then $\Aut^0(X)\cong\Gm$.
\end{lemma}

\begin{proof}
By Lemma~\ref{lemma:dP1-Ga} these surfaces do not admit a $\Ga$-action
and they are not toric because their singularities are not cyclic quotient.
On the~other hand, it is easy to see that each of these surfaces admits a $\Gm$-action.
\end{proof}

We also need the~following easy local fact.

\begin{lemma}
\label{lemma:local}
Let $(X\ni P)$ be a Du Val singularity defined over $\CC$ that contains a
reduced irreducible curve $C$
such that $C$ is a Cartier divisor on $X$, and the~singularity $(C\ni P)$ is a simple cusp.
\begin{enumerate}
\item\label{lemma:local-i}
If $(X\ni P)$ is of type $\type{A}_n$, then $n\leqslant 2$.
\item\label{lemma:local-ii}
If $(X\ni P)$ is of type $\type{D}_n$ with $n\geqslant 5$, then some
small analytic neighborhood of
$(X\ni P)$ can be given by the~equation
$$
x^2+y^2z+z^{n-1}=0,
$$
so that $C$ is cut out by $z=y+\phi(x,y,z)$, where $\mult_0(\phi)\geqslant 2$.
\end{enumerate}
\end{lemma}

\begin{proof}
Let us prove the assertion \ref{lemma:local-i}. In a neighborhood of the point
$P$, the curve $C$ is cut out by a hypersurface, say $H$.
Thus, we have $C=X\cap H$ in $\CC^3$.
Since the multiplicity of the curve $C$ at $P$ equals $2$, the  hypersurface
$H$ is smooth at $P$.
Therefore, we  may assume that $H$ is given by $z=0$, and $C$ is given by
$$
\left\{\aligned
&x^2+y^3=0,\\
&z=0.\\
\endaligned
\right.
$$
Hence, the equation of the surface $X$ is
$$
x^2+y^3+z\phi(x,y,z)=0.
$$
Since $P\in X$ is a point of type $\mathrm{A}_n$, the rank of the quadratic part of this equation is at least $2$.
Then $\phi(x,y,z)$ contains a linear term. This implies that $P\in X$ is of type $\type{A}_2$ or $\type{A}_3$.

Now, let us prove the assertion \ref{lemma:local-ii}.
We may assume that $P\in X$ is given in $\CC^3$ by the equation
$$
x^2+y^2z+z^{n-1}=0.
$$
As above, we have $C=X\cap H$, where $H$ is a hypersurface that is smooth at $P$.
Then the equation of the hypersurface $H$ must contain a linear term.
Moreover, one can see that this equation must be of the form  $z=y+\phi(x,y,z)$,
which implies \ref{lemma:local-ii}.
\end{proof}

\begin{corollary}
\label{corollary:local}
Let $X$ be a surface admitting an effective $\Gm$-action,
let $C$ be a $\Gm$-invariant reduced irreducible curve in $X$ that is a Cartier divisor on $X$,
and let $P$ be its singular point. Suppose that $X$ has Du Val singularity of type $\type{D}_n$ at $P$,
and $C$ has a simple cusp at $P$.
Then~$n=4$.
\end{corollary}

\begin{proof}
Suppose that $n>4$. There exists a $\Gm$-equavariant embedding of the germ
$P\in X$ to $\CC^3$.
Let us choose $\Gm$-semi-invariant  coordinates in $\CC^3$.
By Lemma~\ref{lemma:local}, we see that $C=X\cap H$, where the equation of $H$  has the form  $z=y+\phi(x,y,z)$.
But this equation cannot be $\Gm$-semi-invariant, which is a contradiction.
\end{proof}

Now, we are ready to prove Proposition~\xref{prop:d-1}.

\begin{proof}[Proof of Proposition~\xref{prop:d-1}]

Let $X$ be a~Du Val del Pezzo surface of degree $1$.
Then $X$ does not admit an effective $\Ga$-action by Lemma~\ref{lemma:dP1-Ga}.
Thus, the~group $\Aut(X)$ is infinite if and only if $X$ admits an effective action of the~group $\Gm$.

Suppose that $X$ admits an effective $\Gm$-action.
By Lemma~\ref{lemma:degree-1-Gm-Aut},
to complete the~proof, it is~enough to show that $X$ is one of the~surfaces~\ref{d=1:E8}, \ref{d=1:E7-A1}, \ref{d=1:E6-A2}, \ref{d=1:2D4}.

Let $\Phi\colon X\dasharrow\PP^1$ be the~anticanonical map.
It is $\Gm$-equivariant, and all its~fibers are reduced irreducible curves of arithmetic genus $1$.
Since $\Gm$ cannot effectively act on a smooth elliptic curve, we conclude that $\Gm$ acts non-trivially on the~base of $\Phi$.
There are exactly two $\Gm$-invariant fibers. Denote them by $C_1$ and $C_2$.
Any fiber of $\Phi$ different from $C_1$ and $C_2$ is a smooth elliptic curve.
Thus we have
\[
\uprho(X)+2=\chi(X)=\chi(C_1)+\chi(C_2)-1,
\]
so that $\chi(C_1)+\chi(C_2)=\uprho(X)+3\geqslant 4$.
Since $\chi(C_i)\leqslant 2$, we have $\uprho(X)=1$ and $\chi(C_1)=\chi(C_2)=2$.
This means in particular that $C_1$ and $C_2$ are cuspidal curves of arithmetic genus $1$.

Let $P_i$ be the~singular point of $C_i$. Then
$$
\varnothing \neq\Sing(X)\subset\big\{P_1,\, P_2\big\}.
$$
The singularity of the~surface $X$ at the~point $P_1$ is of type $\type{A}_{n_1}$, $\type{D}_{n_1}$ or $\type{E}_{n_1}$ for some $n_1\geqslant 1$.
Likewise,~if $X$ is singular at $P_2$, then
$P_2$ is a singular point of type $\type{A}_{n_2}$, $\type{D}_{n_2}$ or $\type{E}_{n_2}$ for some~$n_2\geqslant 0$,
where $n_2=0$ simply means that the~point $P_1$ is the~only singular point of the~del Pezzo surface~$X$.
Without loss of generality, we may assume that $n_1\geqslant n_2$.
Since $\uprho(X)=1$, $n_1+n_2=8$ and $n_1\geqslant 4$.
Now, using Corollary~\ref{corollary:local},
we obtain the~following possibilities for $\Type(X)$: $\type{E_8}$, $\type{E_7}\type{A_1}$, $\type{E_6}\type{A_2}$, $2\type{D_4}$.
But $X$ is uniquely determined by $\Type(X)$ and the~fact
that $|-K_X|$ has two singular curves that are both cuspidal.
This follows from \cite[Theorem~1.2]{Ye2002} and \cite[Table~4.1]{Ye2002},
see also Remark~\ref{remark:d-1-E8}.
Thus, we conclude that $X$ is one of the~surfaces~\ref{d=1:E8}, \ref{d=1:E7-A1}, \ref{d=1:E6-A2}, \ref{d=1:2D4} (see e.g. \cite[Satz~2.11]{Brieskorn}).
\end{proof}

Let us conclude this section by an observation that the~surfaces \ref{d=1:E7-A1}, \ref{d=1:E6-A2}, \ref{d=1:2D4}
can be obtained as finite quotients of other surfaces in Big Table.
The~surface \ref{d=1:E7-A1} is the~quotient of the~surface~\ref{d=2:E6}~by~$\mumu_2$.
The~surface \ref{d=1:E6-A2} is the~quotient of~the~cubic surface \ref{d=3:D4} by $\mumu_3$,
and~\ref{d=1:2D4} is the~quotient of a special member of the~family~\ref{d=4:2A1-8lines} by $\mumu_2\times\mumu_2$.
In all the~cases the~action of the~group is free outside the~singular locus.
This observation can be used to obtain the~description of the~surfaces \ref{d=1:E7-A1}, \ref{d=1:E6-A2},~\ref{d=1:2D4}.
To show  this one can look at the~exact sequence
\[
0\longrightarrow \Pic(X) \overset{\alpha}  \longrightarrow \Cl(X) \overset{\beta} \longrightarrow \bigoplus_{P\in X}\Cl(X,P),
\]
where $\Cl(X,P)$ is the~local Weil divisor class group of the point $P\in X$.
The~map $\alpha$ is a primitive embedding. Hence,  we have $\Cl(X)=\Pic(X)\oplus \Cl(X)_{\mathrm{tors}}$.
By Corollary~\ref{corollary:Cl=Z}, we have
$$
\Cl(X)_{\mathrm{tors}}\neq 0.
$$
By Lemma~\ref{lemma:Cl}\ref{lemma:Cl:tors},
the~group $\Cl(X)_{\tors}$ defines a~Galois abelian cover $\pi\colon X^\prime\to X$
which is \'etale outside of the~locus $\Sing(X)$ and whose degree is $|\Cl(X)_{\tors}|$.
Using the~local description of such covers (see \cite{YPG,Brieskorn}),
we see that $\Type(X^\prime)=\type{E}_6$, $\type{D}_4$, $2\type{A}_1$ in the cases  \ref{d=1:E7-A1}, \ref{d=1:E6-A2}, \ref{d=1:2D4}, respectively.

\section{Del Pezzo surfaces of degree $2$}
\label{section:deg-2}

In this section, we prove Main Theorem for del Pezzo surfaces of degree $2$.
To do this, we need one (probably known) result about singular cubic and quartic curves (cf. \cite{Hui,Wall}).

\begin{proposition}
\label{proposition:quartic-curves}
Let $C$ be a reduced cubic or quartic curve in $\PP^2$ such that $\Aut(\PP^2,C)$ is infinite.
Then the~curve $C$ and the~group $\Aut^0(\PP^2,C)$ are given in the~following table:
\newline\begin{center}\renewcommand{\arraystretch}{1.3}\rm
\begin{longtable}{|p{0.8\textwidth}|c|}
\hline
\quad\centering\text{Equation of the~curve $C$ up to the~action of $\PGL_3(\Bbbk)$} \quad&\quad $\Aut^0(\PP^2,C)$ \quad\\
\hline\hline
\centering$x_0x_1(x_0+x_1)=0$&$\Ga^2\rtimes\Gm$
\\
\hline
\centering$x_0x_1x_2=0$&$\Gm^2$
\\
\hline
\centering$x_0(x_0x_2+x_1^2)=0$&$\BB_2$
\\
\hline
\centering$x_0^2x_2+x_1^3=0$&$\Gm$
\\
\hline
\centering$x_1(x_0x_2+x_1^2)=0$&$\Gm$
\\
\hline
\centering$x_0x_1(x_0-x_1)(x_0-\lambda x_1)=0$ for $\lambda\in\Bbbk\setminus\{0,1\}$&$\BB_2$
\\
\hline
\centering$x_0(x_0^2x_2+x_1^3)=0$&$\Gm$
\\
\hline
\centering$x_0x_1(x_0x_2+x_1^2)=0$&$\Gm$
\\
\hline
\centering$(x_0x_2+x_1^2)^2-x_0^4=0$&$\Ga$
\\
\hline
\centering$x_2(x_0^2x_2+x_1^3)=0$&$\Gm$
\\
\hline
\centering$x_0x_1x_2(x_1+x_2)=0$&$\Gm$
\\
\hline
\centering$x_0x_1(x_0x_1+x_2^2)=0$&$\Gm$
\\
\hline
\centering$x_0^3x_2+x_1^4=0$&$\Gm$
\\
\hline
\centering$x_1(x_0^2x_2+x_1^3)=0$&$\Gm$
\\
\hline
\centering$(x_2^2+x_0x_1)(x_2^2+\lambda x_0x_1)=0$ for $\lambda\in\Bbbk\setminus\{0,1\}$&$\Gm$
\\
\hline
\end{longtable}
\end{center}
\end{proposition}

\begin{proof}
If $C$ is one of the~curves in the~table, we can explicitly describe $\Aut^0(\PP^2,C)$ by finding all elements in $\Aut(\PP^2)\cong\PGL_3(\Bbbk)$
that leaves every irreducible component of the~curve $C$ invariant.
For example, if $C$ is given by $x_0x_1(x_0+x_1)=0$, then $\Aut^0(\PP^2,C)$ consists of the~transformations
$$
(x_0:x_1:x_2)\longmapsto\big(tx_0:tx_1:x_2+ax_0+bx_1\big)
$$
for any $t\in\Bbbk^\ast$, $a\in\Bbbk$ and $b\in\Bbbk$. Thus, in this case, we have $\Aut^0(\PP^2,C)\cong\Ga^2\rtimes\Gm$ as required.

Similarly, if $C$ is given by $x_0(x_0x_2+x_1^2)=0$, then $\Aut^0(\PP^2,C)$ consists of the~transformations
$$
(x_0:x_1:x_2)\longmapsto\Big(t^2x_0:tx_1+ax_0:x_2-\frac{2a}{t}x_1-\frac{a^2}{t^2}x_0\Big)
$$
for any $t\in\Bbbk^\ast$ and $a\in\Bbbk$, so that $\Aut^0(\PP^2,C)\cong\Ga\rtimes_{(1)}\Gm$.
Likewise, if $C$ is the~cubic $x_0^2x_2+x_1^3=0$, then $\Aut^0(\PP^2,C)$ consists of the~transformations
$(x_0:x_1:x_2)\longmapsto(t^3x_0:t^2x_1:x_2)$, where $t\in\Bbbk^\ast$.
Thus, in this case, we have $\Aut^0(\PP^2,C)\cong\Gm$.

The computations are very similar in all remaining~cases.
For instance, if $C$ is the~quartic curve that is given by $(x_0x_2+x_1^2)^2-x_0^4=0$, then $\Aut(\PP^2,C)$ consists of the~transformations
$$
(x_0:x_1:x_2)\longmapsto\Big(\zeta^2 x_0:\zeta x_1-\frac{\zeta a}2x_0:x_2+ax_1-\frac{a^2}{4}x_0\Big)
$$
where $\zeta\in\{\pm 1,\pm i\}$ and $a\in\Bbbk$, which implies that $\Aut(\PP^2,C)\cong\Ga\rtimes\mumu_4$, so that $\Aut^0(\PP^2,C)\cong\Ga$.
We leave the~computations in the~remaining cases to the~reader.

Therefore, to complete the~proof, we must show that $C$ is one of the~curves listed in the~table.
If $C$ is the~cubic curve, then $C$ must be singular.
On the~other hand, all singular cubic curves are already listed in the~table except for the~nodal one that is given by
$$
x_2(x_0^2+x_1^2)+x_1^2=0.
$$
However, if $C$ is this curve, then $\Aut(\PP^2,C)$ is finite.
Therefore, we may assume that $\deg(C)=4$. Then we may have the~following five cases:
\begin{enumerate}
\item\label{plane-quadrtic1}
the~curve $C$ is irreducible;
\item\label{plane-quadrtic2}
$C=C_1+C_2$, where $C_1$ is a line and $C_2$ is an irreducible cubic;
\item\label{plane-quadrtic3}
$C=C_1+C_2$, where $C_1$ and $C_2$ are irreducible conics;
\item\label{plane-quadrtic4}
$C=C_1+C_2+C_3$, where $C_1$ and $C_2$ are lines, and $C_2$ are irreducible conic;
\item\label{plane-quadrtic5}
$C=C_1+C_2+C_3+C_3$, where $C_1$, $C_2$, $C_3$ and $C_4$ are lines.
\end{enumerate}
Moreover, by our assumption, the~group $\Aut(\PP^2,C)$ contains a subgroup isomorphic to either $\Ga$ or $\Gm$~(or~both).
We deal with these two (slightly overlapping) possibilities separately.

Suppose that $\Aut(\PP^2,C)\supset \Ga$.
Since each irreducible component is $\Ga$-invariant, we conclude that the~case \ref{plane-quadrtic2} is impossible.
Likewise, if $C$ is irreducible, then $C$ has a $\Ga$-open orbit $U\cong\Aff^1$.
Hence the~curve $C$ must be rational and its normalization morphism must be a homeomorphism, and the~complement $C\setminus U$
is a single point, say $P$.
The projection $C\dashrightarrow \PP^1$ from $P$ must be $\Ga$-equivariant, so it is an isomorphism on $U$.
This implies that $P$ is a triple point and there is exactly one line $L\subset \PP^2$ such that $C\cap L=P$.
We may assume that $P=(0:0:1)$ and $L$ is given by $x_0=0$.
Then the~equation of $C$ has the~form $x_0^3x_2+x_1^4=0$.
But then \mbox{$\Aut^0(\PP^2,C)\cong\Gm$}, which is a contradiction.
Hence, we conclude that case \ref{plane-quadrtic1} is also impossible.

If we are in case \ref{plane-quadrtic3}, then the~$\Ga$-action of each irreducible conic $C_1$ and $C_2$ is effective,
so that the~intersection $C_1\cap C_2$ consists of one point.
In this case, in appropriate projective coordinates, the~curve $C$ is given by
$$
\big(x_1x_2+x_0^2+x_1^2\big)\big(x_1x_2+x_0^2-x_1^2\big)=0,
$$
so that $C$ is listed in the~table as required.

If we are in case \ref{plane-quadrtic4} or case \ref{plane-quadrtic5}, then the~closure of any one-dimensional $\Ga$-orbit is a~line in the~pencil generated by $C_1$ and $C_2$,
which implies that $C$ is~a~union of four lines passing through one point.
Hence, we are in case \ref{plane-quadrtic5}, and the~curve $C$ can be given by
$$
x_0x_1(x_0-x_1)(x_0-\lambda x_1)=0
$$
for some $\lambda\in\Bbbk\setminus\{0,1\}$, so that $C$ is in the~table as well.

To complete the~proof, we may assume that $\Aut(\PP^2,C)\supset \Gm$.
If $C$ is irreducible, then it can be given as the~closure of the~image of the~map $t\mapsto(1:t:t^4)$,
so that $C$ is the~curve $x_0^3x_2-x_1^4=0$ in the~table.
Hence, we may assume that $C$ is reducible, i.e. we are not in case \ref{plane-quadrtic1}.

Suppose that $\Gm$ acts trivially on some irreducible component of the~curve $C$.
This component must be a line, so that we are in one of the~cases \ref{plane-quadrtic2}, \ref{plane-quadrtic4} or \ref{plane-quadrtic5}.
Without loss of generality, we may assume that $\Gm$ acts trivially on the~line $C_1$.
Then there exists a~$\Gm$-fixed point $O\in \PP^2\setminus C_1$,
so~that the~closure of any $\Gm$-orbit in $\PP^2$ is a~line connecting $O$ and a~point in $C_1$.
This implies that we are in case \ref{plane-quadrtic5}, and
the~lines $C_2$, $C_3$, $C_4$ all pass through $O$, so that $C$ is given by
$$
x_0x_1x_2(x_1+x_2)=0
$$
in appropriate projective coordinates. This curve is in the~table.
Therefore, to complete the~proof, we may assume that $\Gm$ acts effectively on each irreducible component of the~curve $C$.

If we are in case \ref{plane-quadrtic5}, then each line among $C_1$, $C_2$, $C_3$ and $C_4$ meet the~union of the~remaining lines in at most two points.
This implies that all these four lines must pass through one point.
Such curve is in the~table and we already met it earlier in the~proof.
Hence, case \ref{plane-quadrtic5} is done.

Suppose that we are in case \ref{plane-quadrtic3}, so that $C=C_1+C_2$, where both $C_1$ and $C_2$ are irreducible conics.
Then the~intersection $C_1\cap C_2$ consists of at most two points.
Moreover, the~intersection cannot consists of one point, since otherwise we would have $\Aut(\PP^2,C)\cong\Ga$.
Hence, we see that the~intersection $C_1\cap C_2$ consists of exactly two points.
Then the~curve $C$ can be given by
$$
(x_2^2+x_0x_1)(x_2^2+\lambda x_0x_1)=0
$$
for some $\lambda\in\Bbbk\setminus\{0,1\}$.
This curve is also in the~table. Thus, case \ref{plane-quadrtic3} is also done.

Now we suppose that we are in case \ref{plane-quadrtic4}. Then $C_1$ and $C_2$ are lines, and $C_3$ is a conic.
Then
$$
\#\Big(C_3\cap\big(C_1\cup C_2\big)\Big)\leqslant 2,
$$
so that at least one of the~lines $C_1$ and $C_2$ must be tangent to $C_3$.
If only one of them is tangent, then $C$ can be given by $x_0x_1(x_0x_2+x_1^2)=0$.
Similarly, if both lines are tangent to the~conic $C_3$,
then $C$ can be given by $x_0x_1(x_0x_1+x_2^2)=0$.
In both subcases, the~curve $C$ is in the~table.

Finally, we suppose that we are in case \ref{plane-quadrtic2}.
Then $C_2$ is a cuspidal cubic curve.
Now, choosing appropriate coordinates on $\PP^2$, we may assume that $C_2$ is given by
$$
x_0^2x_2-x_1^3=0,
$$
and the~$\Gm$-action on $\PP^2$ is described earlier in the~proof.
Then the~$\Gm$-action on $C_2$ has exactly two fixed points: the~points $(0:0:1)$ and $(1:0:0)$.
If the~line $C_1$ passes through both of them, then the~curve $C$ is given by
$$
x_0x_1(x_0x_2+x_1^2)=0.
$$
Similarly, if $(1:0:0)\in C_1$ and $(0:0:1)\notin C_1$, then the~curve $C$ is given by $x_2(x_0^2x_2+x_1^3)=0$.
Vice versa, if $(1:0:0)\notin C_1$ and $(0:0:1)\in C_1$, then the~curve $C$ is given by $x_0(x_0^2x_2+x_1^3)=0$.
In every subcase, we see that the~quartic curve $C$ is listed in the~table as required.
\end{proof}

Using Proposition~\ref{proposition:quartic-curves}, we immediately obtain

\begin{corollary}
\label{corollary:d-2}
Main Theorem holds for del Pezzo surfaces of degree $2$.
\end{corollary}

\begin{proof}
Let $X$ be a~Du Val del Pezzo surface of degree $2$.
Then $X$ is a hypersurface in $\PP(1,1,1,2)$  that is given by
$$
w^2=\phi_4(x_0,x_1,x_2),
$$
where $\phi_4(x_0,x_1,x_2)$ is a homogeneous polynomial of degree $4$.
The natural projection to $\PP^2$ gives a~double cover $\pi\colon X\to \PP^2$.
Let $B$ be the~branch curve of this double cover.
Then $B$ is the~quartic curve in $\PP^2$ that is given by $\phi_4(x_0,x_1,x_2)=0$.

Since the~double cover $\pi$ is $\Aut(X)$-equivariant, it gives a homomorphism
$\Aut^0(X)\to\Aut^0(\PP^2,B)$,
whose kernel is either trivial or isomorphic to $\mumu_2$.
Thus, the~curve $B$ must be one of the~quartic curves listed in the~table in Proposition~\ref{proposition:quartic-curves}
except for the~quartic curve consisting of four lines passing through one point, because $X$ has Du Val singularities.
Now, going through the~equations listed in the~table in Proposition~\ref{proposition:quartic-curves},
we obtain all possibilities for the~polynomial~$\phi_4(x_0,x_1,x_2)$.
This~shows that if $\Aut(X)$ is infinite, then
\begin{itemize}
\item
either $\Aut(X)\cong\Ga$ and $X$ is the~surface~\ref{d=2:A7},
\item
or $\Aut(X)\cong\Gm$ and $X$ is one of the~surfaces
\ref{d=2:E7}, \ref{d=2:D6-A1}, \ref{d=2:A5-A2}, \ref{d=2:D4-3A1}, \ref{d=2:2A3-A1}, \ref{d=2:E6}, \ref{d=2:D5-A1}, \ref{d=2:2A3}.
\end{itemize}

Vice versa, if $X$ is the~surface~\ref{d=2:A7}, then the~group $\Ga$ acts on $X$ as follows:
$$
\big(x_0:x_1:x_2:w\big)\longmapsto\big(x_0+tx_1:x_1:x_2-2tx_0-t^2x_1:w\big),
$$
where $t\in\Ga$. Thus, in this case, we have $\Aut(X)\cong\Ga$. Similarly, if $X$ is one of the~del Pezzo surfaces
\ref{d=2:E7}, \ref{d=2:D6-A1}, \ref{d=2:A5-A2}, \ref{d=2:D4-3A1}, \ref{d=2:2A3-A1}, \ref{d=2:E6}, \ref{d=2:D5-A1}, \ref{d=2:2A3},
then $X$ admits an effective action of the~group $\Gm$, so~that~$\Aut(X)\cong\Gm$ as listed in Big Table.
\end{proof}

\section{Cubic surfaces}
\label{section:cubic-surfaces}

Now, we prove Main Theorem for del Pezzo surfaces of degree $3$, which easily follows from \cite{Sakamaki2010}.
Nevertheless, we prefer to give an independent proof here.

\begin{lemma}
\label{lemma:cubic:Ga-Gm}
Let $X$be a~Du Val cubic surface in $\mathbb{P}^3$.
\begin{enumerate}
\item \label{lemma:cubic:Ga-Gm:1}
If $X$ contains three lines $L_1$, $L_2$, $L_3$ that meet each other at three distinct points \mbox{\textup(a triangle\textup)},
then $X$ does not admit an effective action of the~group $\Ga$.
\item \label{lemma:cubic:Ga-Gm:2}
If $X$ is toric, then $\uprho(X)=1$, $\numl(X)=3$, and the~toric boundary is composed of three lines forming a triangle.
\end{enumerate}
\end{lemma}

\begin{proof}
If $X$ contains a triangle and admits an effective $\Ga$-action,
then  the~$\Ga$-action is trivial on the~triangle,
so that this action is trivial on the~hyperplane in $\PP^3$ that passes through the~triangle,
which implies that the~closure of any $\Ga$-orbit in $X$ is contained in a~line.
The latter is impossible, since $X$ contains finitely many lines.
This proves \ref{lemma:cubic:Ga-Gm:1}

To prove \ref{lemma:cubic:Ga-Gm:2}, suppose that the~surface $X$ is toric.
Let $D=D_1+\cdots+D_r$ be the~toric boundary.
Then $r=\uprho(X)+2$. Since every line on $X$ is torus-invariant and $D\sim -K_X$, we have
$$
3=\sum_{i=1}^r (-K_X)\cdot D_i\geqslant r.
$$
Therefore, we conclude that $\uprho(X)=1$, $r=3$ and $-K_X\cdot D_1=-K_X\cdot D_2=-K_X\cdot D_3=1$.
Moreover, the~lines $D_1$, $D_2$, $D_3$ form a triangle, because the~pair $(X,D)$ has log canonical singularities.
\end{proof}

Now, we are ready to prove

\begin{proposition}
\label{proposition:deg-3-ind-3}
Main Theorem holds for weakly minimal cubic surfaces.
\end{proposition}

\begin{proof}
Let $X$ be a weakly minimal Du Val cubic surface.
If $\ind(X)=1$, then Theorem~\ref{theorem:Fano-Weil-index-big} implies that $X$ is the~surfaces~\ref{d=3:A4-A1},
and its basic properties are described in Example~\ref{ex:A4A1cubic}.
In this case, the~surface $X$ admits an algebraic torus action
\[
(x_0,x_1,x_2,x_3) \longmapsto (x_0,t x_1,t^2 x_2,t^{-1} x_3).
\]
Since $X$ contains three lines passing through one point, it is not toric.
Since $X$ contains a triangle, it does not admit an unipotent group action by Lemma~\ref{lemma:cubic:Ga-Gm},
so that $\Aut^0(X)\cong \Gm$ as required.

Thus, to complete the~proof, we may assume that $\ind(X)>1$.
By Theorem~\ref{theorem:index>1}, we have only the~following possibilities: \ref{d=3:E6}, \ref{d=3:A5-A1}, \ref{d=3:3A2}, \ref{d=3:A5}, \ref{d=3:2A2-A1}, \ref{d=3:2A2}.

Consider the~cases~\ref{d=3:3A2},~\ref{d=3:2A2-A1},~\ref{d=3:2A2}.
From the~equations in Theorem~\ref{theorem:index>1}, we see that $X$ contains a~triangle that is cut out by $y_1y_2=0$.
Then by Lemma~\ref{lemma:cubic:Ga-Gm}\ref{lemma:cubic:Ga-Gm:1},
we~conclude that the~unipotent radical of $\Aut^0(X)$ is trivial.
The~surface~\ref{d=3:3A2} is a toric cubic surface because its equation is binomial,
and the~surfaces~\ref{d=3:2A2-A1} and ~\ref{d=3:2A2} are not toric by Lemma~\ref{lemma:cubic:Ga-Gm}\ref{lemma:cubic:Ga-Gm:2}.
Therefore, if $X$ is the~surface~\ref{d=3:3A2}, then $\Aut^0(X)\cong\Gm^2$.
Similarly, if $X$ is one of the~surfaces~\ref{d=3:2A2-A1} or~\ref{d=3:2A2}, then we have $\Aut^0(X)\cong\Gm$,
because $X$ admits a diagonal effective action of the~group $\Gm$.

Now, we suppose that $X$ is the~surface~\ref{d=3:A5}. Then $\Type(X)=\type{A}_5$,
and it follows from Theorem~\ref{theorem:index>1} that $X$ is a hypersurface in $\PP(1,2,3,3)$ that is given~by
$$
y_3^2=y_2^3+y_1^6+y_1y_2y_3'.
$$
Let $L_1$, $L_2$, $L_3$ be the~curves $y_1=y_3^2-y_2^3=0$, $y_2=y_3-y_1^3=0$, $y_2=y_3+y_1^3=0$, respectively.
Then~$L_1$, $L_2$ and $L_3$ are lines meeting at one point.
If $X$ admits an effective action of the~group~$\Gm$, then $L_3$ contains a $\Gm$-fixed point $P\notin\Sing(X)$,
and there exists a $\Gm$-equivariant diagram
$$
\xymatrix@R=0.8em{
&\widehat{X}\ar@{->}[rd]^{\beta}\ar@{->}[dl]_{\alpha}&\\%
X&&Y}
$$
where $\alpha$ is the~blow up of the~point $P$, the~morphism $\beta$ is the~birational contraction of the~proper transform of the~line $L_3$,
and $Y$ is a singular del Pezzo surface of degree $2$ such that $\Type(X)=\type{A}_6$.
The~latter contradicts Corollary~\ref{corollary:d-2}, so that $X$ does not admit an effective action of the~group~$\Gm$.
Then $\Aut^0(X)$ is unipotent.
By Corollary~\ref{corollary:Ga}, the~surface $X$ does not admit an effective~$\Ga^2$-action.
Then $\dim \Aut^0(X)\leqslant 1$. On the~other hand, the~group $\Aut(X)$ contains transformations
$$
\big(y_1:y_2:y_3:y_3'\big)\longmapsto
\big(y_1:y_2:y_3+ay_1y_2:y_3'+2ay_3+a^2y_1y_2\big),
$$
where $a\in\Bbbk$. They generate a group isomorphic to $\Ga$. Then $\Aut^0(X)\cong\Ga$ by Corollary~\ref{corollary:connected-group-rational-surface}.

Let $X$ be the~surface~\ref{d=3:A5-A1}.
As in the~previous case, the~surface $X$ is a sextic hypersurface in the~weighted projective space $\PP(1,2,3,3)$.
But now the~surface $X$ is given by $y_3^2=y_2^3+y_3'y_1y_2$.
Observe that the~group $\Aut(X)$ contains transformations
$$
(y_1:y_2:y_3:y_3')\longmapsto\big(y_1:t^2y_2:t^3y_3+at^3y_1y_2:t^4y_3'+2at^4y_3+a^2t^4y_1y_2\big)
$$
for any $a\in\Bbbk$ and $t\in\Bbbk^\ast$.
They generate a group isomorphic to~$\BB_2$.
This implies that $\Aut^0(X)\cong\BB_2$,
because $X$ is not toric by Lemma~\ref{lemma:cubic:Ga-Gm}\ref{lemma:cubic:Ga-Gm:2},
and $X$ admits no effective~$\Ga^2$-action by Corollary~\ref{corollary:Ga}.

Finally, if $X$ is the~surface~\ref{d=3:E6}, then it follows from Theorem~\ref{theorem:index>1}
that $X$ is a hypersurface in $\PP(1,2,3,3)$ that is given by $y_3^2=y_2^3+y_3'y_1^3$.
Using this, one can show that the~group $\Aut^0(X)$ consists of transformations
\[
\big(y_1:y_2:y_3:y_3'\big)\longmapsto\big(y_1:t^2y_2:t^3y_3+ay_1^3:t^6y_3'+a^2y_1^3+2at^3y_3\big),
\]
where $a\in\Bbbk$ and $t\in\Bbbk^\ast$. Thus, in this case, we have $\Aut^0(X)\cong\Ga\rtimes_{(3)}\Gm$,
which also follows from Corollary~\ref{Cor:Cl=Z:Aut}.
\end{proof}

To complete the~proof of Main Theorem for Du Val cubic surfaces, we need

\begin{lemma}
\label{lemma:degree-3-Gm}
Let $X$ be a non-weakly minimal Du Val cubic surface such that $\Aut^0(X)$ is infinite.
Then $\Aut^0(X)\cong\Gm$ and $X$ is one of the~surfaces~\ref{d=3:D5}, \ref{d=3:A3-2A1} or~\ref{d=3:D4}.
\end{lemma}

\begin{proof}
The surface $X$ contains a~line $L$ such that $L\subset X\setminus\Sing(X)$.
By Lemma~\ref{lemma:conic-bundles}\ref{conic:bundle:3} there is a~conic bundle $\psi\colon X\to\PP^1$ such that $L$ is its double section.
If $X$ admits an effective $\Ga$-action,
then the~group $\Ga$ fixes the~ramification points of the~double cover $L\to\PP^1$ induced by $\psi$,
so that the~group $\Ga$ acts trivially on~$L$, which implies that it also acts trivially on the~fibers of the~conic bundle $\psi$,
so that the~$\Ga$-action on $X$ is trivial, which is a~contradiction.
Hence, we~conclude~that the~group $\Aut^0(X)$ contains no unipotent subgroups.
Then $\Aut^0(X)$ must be a torus,
which implies that $\Aut^0(X)\cong\Gm$, because $X$ is not toric by Lemma\ref{lemma:cubic:Ga-Gm}\ref{lemma:cubic:Ga-Gm:2}.

Let $\psi^\prime\colon X\to X^\prime$ be the~contraction of the~line $L$.
Then $X^\prime$ is a quartic Du Val del Pezzo surface such that
$\uprho(X^\prime)=\uprho(X)-1$ and $\Type(X)=\Type(X^\prime)$.
Note that the~group $\Aut(X^\prime)$ is infinite,
and $\Aut^0(X)$ is the~stabilizer  in $\Aut^0(X^\prime)$ of the~point $\psi^\prime(L)$.
Let $U^\prime$ be the~complement in $X^\prime$ to the~union of all lines.
Then $\psi^\prime(L)\in U^\prime$ by Corollary~\ref{corollary:extremal-contractionsDP}.

Suppose that $\uprho(X)=2$. Then $\psi$ is an extremal contraction.
By Lemma~\ref{extremal-contractions}\ref{extremal-contractions1},
the~singular points of the~surface $X$ can be of types $\type{D}_4$, $\type{D}_5$, $\type{A}_3$, and $\type{A}_1$, where
$\type{A}_1$ appears even number~of~times.
We have two possibilities: $\Type(X)=\type{D}_5$ and $\type{A}_32\type{A}_1$,
so that $X^\prime$ is one of the~surfaces \ref{d=4:D5} or \ref{d=4:A3-2A1}.
In both cases, the~subset $U^\prime$ is the~open $\Aut^0(X^\prime)$-orbit (cf. Remark~\ref{remark:Cl1:A2}),
which immediately implies that $\Aut^0(X)\cong \Gm$ and $X$ is one of the~surface~\ref{d=3:D5} or \ref{d=3:A3-2A1}.

Now, we assume that $\uprho(X)>2$. If $\Type(X)=\type{D}_4$, then $X^\prime$ is the~surface \ref{d=4:D4}.
Arguing as above, we see that $\Aut^0(X)\cong \Gm$ and $X$ is the~surface~\ref{d=3:D4}.
Thus, to complete the~proof, we may assume that $\uprho(X)>2$ and all the~singularities of $X$ are of type $\type{A}_n$.

We claim that the~action of $\Aut^0(X)$ on  $L$ is trivial.
Indeed, suppose that this is not~the~case.
Let~us seek for a contradiction. Let~$P_1$ and $P_2$ be the~ramification points of the~double cover~$L\to\PP^1$,
let $F_1$ and $F_2$ be the~fibers of the~conic bundle $\psi$ passing through the~points $P_1$ and $P_2$,~respectively.
Then~$P_1$ and $P_2$ are fixed by $\Aut^0(X)$, and these are all $\Aut^0(X)$-fixed points on $L$,
so that all fibers of the~conic bundle $\psi$ other than $F_1$ and $F_2$ are smooth.
Since $\uprho(X)>2$, there exists at least one reducible fiber.
Thus, we may assume that $F_1$ is reducible. Then
$$
F_1=F^\prime_1+F^{\prime\prime}_1,
$$
where $F_1^\prime$ and $F_1^{\prime\prime}$ are lines.
Then $P_1=F_1^\prime\cap F_1^{\prime\prime}\cap L$ and the~surface $X$ is smooth along $F_1$,
which implies that $\Sing(X)\subset F_2$ and $F_2$ is irreducible (but multiple).
In particular, we have $\uprho(X)=3$.
On the~other hand, using Lemma~\ref{extremal-contractions}\ref{extremal-contractions1}, we see that either $\Type(X)=2\type{A}_1$ or $\Type(X)=\type{A}_3$.
This contradicts the~Noether formula.
Therefore, the~action of $\Aut^0(X)$ on the~line $L$ is trivial,
so~that the~action of the~group $\Aut^0(X)$ on the~base of the~conic bundle $\psi$ is trivial as well.

Let $M^\prime$ be a line in $X^\prime$ (it does exist since $\uprho(X^\prime)>1$),
and let $M$ be its proper transform on~$X$.
Then~$\psi^\prime(L)\not\in M^\prime$ by Corollary~\ref{corollary:extremal-contractionsDP},
so that $M$ is a line on $X$, which is disjoined from the~line~$L$.
Then $M$ is an $\Aut^0(X)$-invariant curve, which is not contained in the~fibers of the~conic bundle $\psi$,
since $\psi$ is given by the~projection from $L$.
Therefore, if $C$ is a general fiber of~$\psi$, then $C$~contains at least three $\Aut^0(X)$-fixed points $C\cap (L\cup M)$,
so that the~$\Aut^0(X)$-action on $C$ must be trivial.
This implies that the~action of $\Aut^0(X)$ on $X$ is also trivial, which is a contradiction.
\end{proof}

Combining Proposition~\ref{proposition:deg-3-ind-3} and Lemma \ref{lemma:degree-3-Gm}, we obtain

\begin{corollary}
\label{corollary:d-3}
Main Theorem holds for del Pezzo surfaces of degree $3$.
\end{corollary}

Thus
Main Theorem holds for Du Val del Pezzo surface of degrees $1$, $2$, $3$, $4$, $5$, $6$.
This follows from Proposition~\ref{prop:d-1} and Corollaries~ \ref{cor:high-degree}, \ref{corollary:d-2}, \ref{corollary:d-3}.
This completes the~proof of Main Theorem.

\newpage
\begin{landscape}%{adjustbox}{angle=90}
\section{Big Table}
\label{section:tables}

Let $X$ be a~Du Val del Pezzo surface such that $\Aut(X)$ is infinite.
Then the~type $\Type(X)$, the~degree $K_X^2$, the~Picard rank~$\uprho(X)$, the~number of lines~$\numl(X)$,
the~Fano--Weil index $\ind(X)$, the~group $\Aut^0(X)$, and the~equation of the~surface $X$ are given below.
The~column No indicates a~del Pezzo surface from which $X$ can be obtained by blowing up a smooth point that
does not lie on a line.

\renewcommand{\arraystretch}{1.13}
\begin{center}
\begin{longtable}{|c|c|c|c|c|c|c|c|p{0.5\textwidth}|c|}\hline
\label{table:main}
&$K_X^2$&$\uprho(X)$&$\numl(X)$&$\Type(X)$&$\ind(X)$&No&$\Aut^0(X)$&\multicolumn2{c|}{equation \& total space}
\\
\hhline{|=|=|=|=|=|=|=|=|=|=|}
\endhead\hline
&$K_X^2$&$\uprho(X)$&$\numl(X)$&$\Type(X)$&$\ind(X)$&No&$\Aut^0(X)$&\multicolumn2{c|}{equation \& total space}
\\
\hhline{|=|=|=|=|=|=|=|=|=|=|}
\endfirsthead
\hline
\no\label{d=1:E8}
&$1$&$1$&$1$&$\type{E}_8$&$1$&--&$\Gm$&\centering$y_3^2=y_2^3+y_1'y_1^5$&$\PP(1,1,2,3)$
\\
\hline
\no\label{d=1:E7-A1}
&$1$&$1$&$3$&$\type{E}_{7}\type{A}_1$&$1$&--&$\Gm$&\centering$y_3^2=y_1^3y_1'y_2+y_2^3$&$\PP(1,1,2,3)$
\\
\hline
\no\label{d=1:E6-A2}
&$1$&$1$&$4$&$\type{E}_6\type{A}_2$&$1$&--&$\Gm$&\centering$y_3^2=y_2^3+y_1'^2y_1^4$&$\PP(1,1,2,3)$
\\
\hline
\no\label{d=1:2D4}
&$1$&$1$&$5$&$2\type{D}_4$&$1$&--&$\Gm$&\centering$y_3^2=y_2(y_2+y_1y_1')(y_2+\lambda y_1y_1')$ for $\lambda\in\Bbbk\setminus\{0,1\}$
&$\PP(1,1,2,3)$
\\
\hline
\no\label{d=2:E7}
&$2$&$1$&$1$&$\type{E}_7$&$2$&--&$\Gm$&\centering$y_2^2=y_1(y_1^2y_1''+y_1'^3)$&$\PP(1,1,1,2)$
\\
\hline
\no\label{d=2:D6-A1}
&$2$&$1$&$2$&$\type{D}_{6}\type{A}_1$&$2$&--&$\Gm$&\centering$y_2^2=y_1y_1'(y_1y_1''+y_1'^2)$&$\PP(1,1,1,2)$
\\
\hline
\no\label{d=2:A7}
&$2$&$1$&$2$&$\type{A}_7$&$1$&--&$\Ga$&\centering$y_2^2=(y_1y_1''+y_1'^2)^2-y_1^4$&$\PP(1,1,1,2)$
\\
\hline
\no\label{d=2:A5-A2}
&$2$&$1$&$3$&$\type{A}_5\type{A}_2$&$2$&--&$\Gm$&\centering$y_2^2=y_1''(y_1^2y_1''+y_1'^3)$&$\PP(1,1,1,2)$
\\
\hline
\no\label{d=2:D4-3A1}
&$2$&$1$&$4$&$\type{D}_43\type{A}_1$&$2$&--&$\Gm$&\centering$y_2^2=y_1y_1'y_1''(y_1'+y_1'')$&$\PP(1,1,1,2)$
\\
\hline
\no\label{d=2:2A3-A1}
&$2$&$1$&$4$&$2\type{A}_3\type{A}_1$&$2$&--&$\Gm$&\centering$y_2^2=y_1y_1'(y_1y_1'+y_1''^2)$&$\PP(1,1,1,2)$
\\
\hline
\no\label{d=2:E6}
&$2$&$2$&$4$&$\type{E}_6$&$1$&\ref{d=3:E6}&$\Gm$&\centering$y_2^2=y_1^3y_1''+y_1'^4$&$\PP(1,1,1,2)$
\\
\hline
\no\label{d=2:D5-A1}
&$2$&$2$&$5$&$\type{D}_5\type{A}_1$&$2$&--&$\Gm$&\centering$y_2^2=y_1'(y_1^2y_1''+y_1'^3)$&$\PP(1,1,1,2)$
\\
\hline
\no\label{d=2:2A3}
&$2$&$2$&$6$&$2\type{A}_3$&$1$&--&$\Gm$&\centering$y_2^2=(y_1''^2+y_1y_1')(y_1''^2+\lambda y_1y_1')$ for $\lambda\in\Bbbk\setminus\{0,1\}$
&$\PP(1,1,1,2)$
\\
\hline
\no\label{d=3:E6}
&$3$&$1$&$1$&$\type{E}_6$&$3$&--&$\Ga\rtimes_{(3)}\Gm$&\centering$x_0x_2^2=x_1^3+x_3x_0^2$&$\PP^3$
\\
\hline
\no\label{d=3:A5-A1}
&$3$&$1$&$2$&$\type{A}_5\type{A}_1$&$3$&--&$\BB_2$&\centering$x_0x_2^2=x_1^3+x_0x_3x_1$&$\PP^3$
\\
\hline
\no\label{d=3:3A2}
&$3$&$1$&$3$&$3\type{A}_2$&$3$&--&$\Gm^2$&\centering$x_0x_2x_3=x_1^3$&$\PP^3$
\\
\hline
\no\label{d=3:D5}
&$3$&$2$&$3$&$\type{D}_5$&$1$&\ref{d=4:D5}&$\Gm$&\centering$x_0^2x_3=x_2(x_0x_2-x_1^2)$&$\PP^3$
\\
\hline
\no\label{d=3:A5}
&$3$&$2$&$3$&$\type{A}_5$&$3$&--&$\Ga$&\centering$x_0x_2^2=x_1^3+x_0^3+x_0x_3x_1$&$\PP^3$
\\
\hline
\no\label{d=3:A4-A1}
&$3$&$2$&$4$&$\type{A}_4\type{A}_1$&$1$&--&$\Gm$&\centering$x_3(x_0x_2-x_1^2)=x_0^2x_1$&$\PP^3$
\\
\hline
\no\label{d=3:A3-2A1}
&$3$&$2$&$5$&$\type{A}_32\type{A}_1$&$1$&\ref{d=4:A3-2A1}&$\Gm$&\centering$x_3(x_0x_2-x_1^2)=x_0x_1^2$&$\PP^3$
\\
\hline
\no\label{d=3:2A2-A1}
&$3$&$2$&$5$&$2\type{A}_2\type{A}_1$&$3$&--&$\Gm$&\centering$x_0x_2x_3=x_1^3+x_0x_1^2$&$\PP^3$
\\
\hline
\no\label{d=3:D4}
&$3$&$3$&$6$&$\type{D}_4$&$1$&\ref{d=4:D4}&$\Gm$&\centering$x_0^2x_3=x_1x_2(x_1+x_2)$&$\PP^3$
\\
\hline
\no\label{d=3:2A2}
&$3$&$3$&$7$&$2\type{A}_2$&$3$&--&$\Gm$&\centering$x_0x_2x_3=x_1(x_1-x_0)(x_1-\lambda x_0)$ for $\lambda\in\Bbbk\setminus\{0,1\}$&$\PP^3$
\\
\hline
\no\label{d=4:D5}
&$4$&$1$&$1$&$\type{D}_5$&$4$&--&$\Ga^2\rtimes\Gm$&\centering$y_3^2=y_2^3+y_1^2y_4$&$\PP(1,2,3,4)$
\\
\hline
\no\label{d=4:A3-2A1}
&$4$&$1$&$2$&$\type{A}_32\type{A}_1$&$4$&--&$\BB_2\times\Gm$&\centering$y_3^2=y_2y_4$&$\PP(1,2,3,4)$
\\
\hline
\no\label{d=4:D4}
&$4$&$2$&$2$&$\type{D}_4$&$2$&--&$\Ga\rtimes_{(2)}\Gm$&\centering$y_2^2=y_2'y_1^2+y_1'^4$&$\PP(1,1,2,2)$
\\
\hline
\no\label{d=4:A4}
&$4$&$2$&$3$&$\type{A}_4$&$1$&\ref{d=5:A4}&$\BB_2$&\centering$x_0x_1-x_2x_3=x_0x_4+x_1x_2+x_3^2=0$&$\PP^4$
\\
\hline
\no\label{d=4:A3-A1}
&$4$&$2$&$3$&$\type{A}_3\type{A}_1$&$4$&--&$\BB_2$&\centering$y_3^2=y_1^6+y_2y_4$&$\PP(1,2,3,4)$
\\
\hline
\no\label{d=4:A2-2A1}
&$4$&$2$&$4$&$\type{A}_22\type{A}_1$&$2$&--&$\Gm^2$&\centering$y_2y_2'=y_1^3y_1'$&$\PP(1,1,2,2)$
\\
\hline
\no\label{d=4:4A1}
&$4$&$2$&$4$&$4\type{A}_1$&$2$&--&$\Gm^2$&\centering$y_2y_2'=y_1^2y_1'^2$&$\PP(1,1,2,2)$
\\
\hline
\no\label{d=4:A3-4lines}
&$4$&$3$&$4$&$\type{A}_3$&$2$&--&$\Ga$&\centering$y_2^2=y_2'y_1y_1'+y_1^4+y_1'^4$&$\PP(1,1,2,2)$
\\
\hline
\no\label{d=4:A3-5lines}
&$4$&$3$&$5$&$\type{A}_3$&$1$&\ref{d=5:A3}&$\Gm$&\centering$x_0x_1-x_2x_3=x_0x_3+x_2x_4+x_1x_3=0$&$\PP^4$
\\
\hline
\no\label{d=4:A2-A1}
&$4$&$3$&$6$&$\type{A}_2\type{A}_1$&$1$&\ref{d=5:A2-A1}&$\Gm$&\centering$x_0x_1-x_2x_3=x_1x_2+x_2x_4+x_3x_4=0$&$\PP^4$
\\
\hline
\no\label{d=4:3A1}
&$4$&$3$&$6$&$3\type{A}_1$&$2$&--&$\Gm$&\centering$y_2y_2'=y_1^2y_1'(y_1'+y_1)$&$\PP(1,1,2,2)$
\\
\hline
\no\label{d=4:2A1-8lines}
&$4$&$4$&$8$&$2\type{A}_1$&$2$&--&$\Gm$&\centering$y_2y_2'=y_1y_1'(y_1'-y_1)(y_1'-\lambda y_1)$
for $\lambda\in\Bbbk\setminus\{0,1\}$&$\PP(1,1,2,2)$
\\
\hline
\no\label{d=5:A4}
&$5$&$1$&$1$&$\type{A}_4$&$5$&--&$\UU_3\rtimes\Gm$&\centering$y_3^2+y_2^3+y_1y_5=0$&$\PP(1,2,3,5)$
\\
\hline
\no\label{d=5:A3}
&$5$&$2$&$2$&$\type{A}_3$&$1$&--&$\Ga^2\rtimes\Gm$&\centering$u_2^2v_0+(u_0^2+u_1u_2)v_1=0$&$\PP^2\times\PP^1$
\\
\hline
\no\label{d=5:A2-A1}
&$5$&$2$&$3$&$\type{A}_2\type{A}_1$&$1$&\ref{d=6:A2-A1}&$\BB_2\times\Gm$&&
\\
\hline
\no\label{d=5:A2}
&$5$&$3$&$4$&$\type{A}_2$&$1$&\ref{d=6:A2}&$\BB_2$&\centering$u_0u_1v_0+(u_1^2+u_0u_2)v_1=0$&$\PP^2\times\PP^1$
\\
\hline
\no\label{d=5:2A1}
&$5$&$3$&$5$&$2\type{A}_1$&$1$&\ref{d=6:2A1}&$\Gm^2$&\centering$u_0^2v_0+u_1u_2v_1=0$&$\PP^2\times\PP^1$
\\
\hline
\no\label{d=5:A1}
&$5$&$4$&$7$&$\type{A}_1$&$1$&\ref{d=6:A1-3l},\ref{d=6:A1-4l}&$\Gm$&\centering$u_0u_1v_0+(u_0+u_1)u_2v_1=0$&$\PP^2\times\PP^1$
\\
\hline
\no\label{d=6:A2-A1}
&$6$&$1$&$1$&$\type{A}_2\type{A}_1$&$6$&--&$\BB_3$&\centering---&$\PP(1,2,3)$
\\
\hline
\no\label{d=6:A2}
&$6$&$2$&$2$&$\type{A}_2$&$3$&--&$\UU_3\rtimes\Gm$&\centering$y_1y_3=y_2^2+y_1'^4$&$\PP(1,1,2,3)$
\\
\hline
\no\label{d=6:2A1}
&$6$&$2$&$2$&$2\type{A}_1$&$2$&--&$\BB_2\times\BB_2$&\centering$y_1y_2=y_1'^2y_1''$&$\PP(1,1,1,2)$
\\
\hline
\no\label{d=6:A1-3l}
&$6$&$3$&$3$&$\type{A}_1$&$2$&--&$\Ga^2\rtimes\Gm$&\centering$y_1y_2=y_1'y_1''(y_1'+y_1'')$&$\PP(1,1,1,2)$
\\
\hline
\no\label{d=6:A1-4l}
&$6$&$3$&$4$&$\type{A}_1$&$1$&\ref{d=7}&$\BB_2\times\Gm$&\centering$u_0v_0+u_1v_1+u_2v_2=0$, $u_0v_1+u_1v_2=0$&$\PP^2\times\PP^2$
\\
\hline
\no\label{d=6:smooth}
&$6$&$4$&$ 6$&smooth&$1$&\ref{d=7:smooth}&$\Gm^2$&\centering$u_0v_0w_0=u_1v_1w_1$&$\PP^1\times\PP^1\times\PP^1$
\\
\hline
\no\label{d=7}
&$7$&$2$&$2$&$\type{A}_1$&$1$&\ref{d=8}&$\BB_3$&&
\\
\hline
\no\label{d=7:smooth}
&$7$&$3$&$3$&smooth&$1$&\ref{d=8:F1},\ref{d=8:P1-P1}&$\BB_2\times\BB_2$&\centering & \\
\hline
\no\label{d=8}
&$8$&$1$&$0$&$\type{A}_1$&$4$&--&$\Ga^3\rtimes(\GL_2(\Bbbk)/\mumu_2)$&\centering---&$\PP(1,1,2)$
\\
\hline
\no\label{d=8:F1}\centering
&$8$&$2$&$1$&smooth&$1$&\ref{d=9:P2}&$\Ga^2\rtimes\GL_2(\Bbbk)$&\centering $u_0v_0=u_1v_1$&$\PP^2\times\PP^1$
\\
\hline
\no\label{d=8:P1-P1}
&$8$&$2$&$0$&smooth&$2$&--&$\PGL_2(\Bbbk)\times\PGL_2(\Bbbk)$&\centering---&$\PP^1\times\PP^1$
\\
\hline
\no\label{d=9:P2}
&$9$&$1$&$0$&smooth&$3$&--&$\PGL_3(\Bbbk)$&\centering---&$\PP^2$
\\
\hline
\end{longtable}
\end{center}
\end{landscape}%{adjustbox}

\newpage
\appendix

\section{Lines and dual graphs}
\label{section:appendix}

In this appendix, we present equations of the~lines on del Pezzo surfaces that appear in Big Table.
We also present the~dual graphs of all the~curves with negative self-intersection numbers on their minimal~resolutions.
As~in~the~paper \cite{Coray1988}, we will denote a~$(-1)$-curve by $\bullet$, and we will denote a~$(-2)$-curve by $\circ$.
We~will exclude surfaces~\ref{d=7} (see Example~\ref{example:dP7-A1}), \ref{d=7:smooth}, \ref{d=8}, \ref{d=8:F1}, \ref{d=8:P1-P1}, \ref{d=9:P2}.

Let $X$ be a~del Pezzo surface of degree $d$ in Big Table.
Take the~equation of $X$ from Big~Table.
The lines on $X$ and the~dual graphs of all the~curves with negative self-intersection numbers on its minimal resolution can be described as follows.

\begin{itemize}
\item[(\ref{d=1:E8})]
One has $d=1$ and $\Type(X)=\type{E}_8$.
The~dual graph is
$$
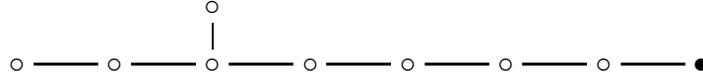
\xymatrix@R=0.8em{
&&\circ\ar@{-}[d]&&&&&\\
\circ\ar@{-}[r]&\circ\ar@{-}[r]&\circ\ar@{-}[r]&\circ\ar@{-}[r]&\circ&\circ\ar@{-}[l]&\circ\ar@{-}[l]&\bullet\ar@{-}[l]}
$$
The surface $X$ is weakly-minimal. The only line on $X$ is $x_3^2-x_2^3=x_0=0$.

\item[(\ref{d=1:E7-A1})]
One has $d=1$ and $\Type(X)=\type{E}_7\type{A}_1$.
The~dual graph is
$$
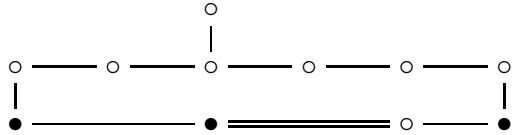
\xymatrix@R=0.8em{
&&\circ\ar@{-}[d]&&&\\
\circ\ar@{-}[r]&\circ\ar@{-}[r]&\circ\ar@{-}[r]&\circ\ar@{-}[r]&\circ&\circ\ar@{-}[l]\\
\bullet\ar@{-}[u]\ar@{-}[rr]&&\bullet\ar@{=}[rr]&&\circ\ar@{-}[r]&\bullet\ar@{-}[u]}
$$
The surface $X$ is weakly-minimal. The lines are cut out by $x_0=0$, $x_1=0$ and $x_3=x_2=0$.

\item[(\ref{d=1:E6-A2})]
One has $d=1$ and $\Type(X)=\type{E}_6\type{A}_2$.
The~dual graph is
$$
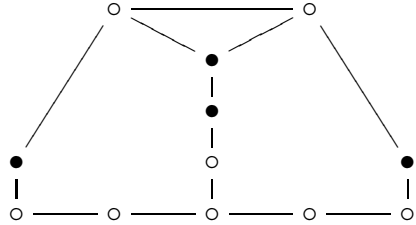
\xymatrix@R=0.6em{
&\circ\ar@{-}[dr]\ar@{-}[rr]\ar@{-}[dddl]&&\circ\ar@{-}[dl]\ar@{-}[dddr]&\\
&&\bullet\ar@{-}[d]&&\\
&&\bullet\ar@{-}[d]&&\\
\bullet\ar@{-}[d]&&\circ\ar@{-}[d]&&\bullet\ar@{-}[d]\\
\circ\ar@{-}[r]&\circ\ar@{-}[r]&\circ\ar@{-}[r]&\circ\ar@{-}[r]&\circ}
$$
The surface $X$ is weakly-minimal. The lines are cut out by $x_0=0$, $x_1=0$ and $x_2=0$.

\item[(\ref{d=1:2D4})]
One has $d=1$ and $\Type(X)=2\type{D}_4$. The~dual graph is
$$
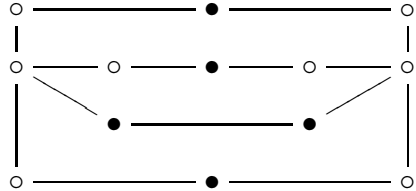
\xymatrix@R=0.8em{
\circ&&\bullet\ar@{-}[ll]\ar@{-}[rr]&&\circ\\
\circ\ar@{-}[u]&\circ\ar@{-}[l]&\bullet\ar@{-}[r]\ar@{-}[l]&\circ\ar@{-}[r]&\circ\ar@{-}[u]\\
&\bullet\ar@{-}[rr]\ar@{-}[ul]&&\bullet\ar@{-}[ur]&\\
\circ\ar@{-}[uu]&&\bullet\ar@{-}[ll]\ar@{-}[rr]&&\circ\ar@{-}[uu]}
$$
The surface $X$ is weakly-minimal. The lines are cut out by $x_0=0$, $x_1=0$ and $x_3=0$.

\item[(\ref{d=2:E7})]
One has $d=2$ and $\Type(X)=\type{E}_{7}$.
The~dual graph is
$$
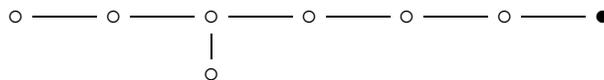
\xymatrix@R=0.8em{
\circ\ar@{-}[r]&\circ\ar@{-}[r]&\circ\ar@{-}[r]&\circ\ar@{-}[r]&\circ\ar@{-}[r]&\circ&\bullet\ar@{-}[l]\\
&&\circ\ar@{-}[u]&&&&}
$$
The surface $X$ is weakly-minimal, and the~line is $x_0=x_3=0$.

\item[(\ref{d=2:D6-A1})]
One has $d=2$ and $\Type(X)=\type{D}_{6}\type{A}_1$.
The~dual graph is
$$
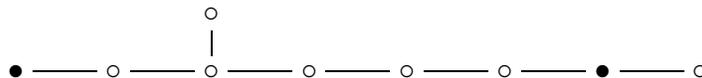
\xymatrix@R=0.8em{
&&\circ\ar@{-}[d]&&&&&\\
\bullet&\circ\ar@{-}[r]\ar@{-}[l]&\circ\ar@{-}[r]&\circ\ar@{-}[r]&\circ\ar@{-}[r]&\circ&\bullet\ar@{-}[l]\ar@{-}[r]&\circ}
$$
The surface $X$ is weakly-minimal. The lines are $x_0=x_3=0$ and $x_1=x_3=0$.

\item[(\ref{d=2:A7})]
One has $d=2$ and $\Type(X)=\type{A}_7$.
The~dual graph is
$$
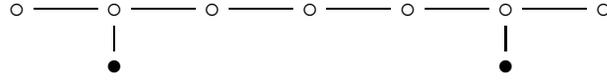
\xymatrix@R=0.8em{
\circ\ar@{-}[r]&\circ\ar@{-}[r]&\circ\ar@{-}[r]&\circ\ar@{-}[r]&\circ\ar@{-}[r]&\circ\ar@{-}[r]&\circ&\\
&\bullet\ar@{-}[u]&&&&\bullet\ar@{-}[u]&&}
$$
The surface $X$ is weakly-minimal, and the~lines are $x_1=x_3\pm x_0^2=0$.

\item[(\ref{d=2:A5-A2})]
One has $d=2$ and $\Type(X)=\type{A}_5\type{A}_2$.
The~dual graph is
$$
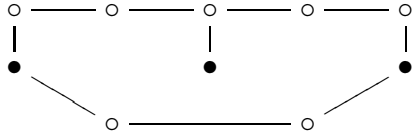
\xymatrix@R=0.8em{
\circ\ar@{-}[r]&\circ\ar@{-}[r]&\circ\ar@{-}[r]&\circ\ar@{-}[r]&\circ\\
\bullet\ar@{-}[u]&&\bullet\ar@{-}[u]&&\bullet\ar@{-}[u]\\
&\circ\ar@{-}[rr]\ar@{-}[ul]&&\circ\ar@{-}[ur]&}
$$
The surface $X$ is weakly-minimal. The lines are $x_1=x_3\pm x_0x_2=0$ and $x_2=x_3=0$.

\item[(\ref{d=2:D4-3A1})]
One has $d=2$ and $\Type(X)=\type{D}_43\type{A}_1$.
The~dual graph is
$$
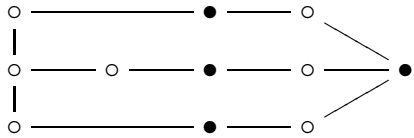
\xymatrix@R=0.8em{
\circ\ar@{-}[rr]\ar@{-}[d]&&\bullet\ar@{-}[r]&\circ\ar@{-}[rd]&\\
\circ\ar@{-}[r]&\circ\ar@{-}[r]&\bullet\ar@{-}[r]&\circ\ar@{-}[r]&\bullet\\
\circ\ar@{-}[rr]\ar@{-}[u]&&\bullet\ar@{-}[r]&\circ\ar@{-}[ur]&}
$$
The surface $X$ is weakly-minimal. The lines are cut out by $x_3=0$.

\item[(\ref{d=2:2A3-A1})]
One has $d=2$ and $\Type(X)=2\type{A}_3\type{A}_1$.
The~dual graph is
$$
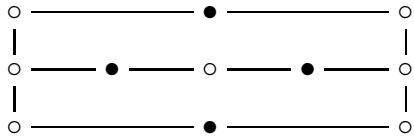
\xymatrix@R=0.8em{
\circ&&\bullet\ar@{-}[ll]\ar@{-}[rr]&&\circ\\
\circ\ar@{-}[u]&\bullet\ar@{-}[l]&\circ\ar@{-}[r]\ar@{-}[l]&\bullet\ar@{-}[r]&\circ\ar@{-}[u]\\
\circ\ar@{-}[u]&&\bullet\ar@{-}[ll]\ar@{-}[rr]&&\circ\ar@{-}[u]}
$$
This surface is weakly-minimal. The lines are $x_0=x_3=0$, $x_1=x_3=0$, $x_2=x_3\pm x_0x_1=0$.

\item[(\ref{d=2:E6})]
One has $d=2$ and $\Type(X)=\type{E}_6$.
The~dual graph is
$$
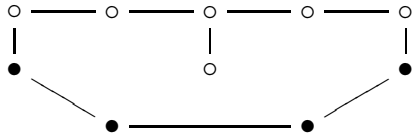
\xymatrix@R=0.8em{
\circ\ar@{-}[r]&\circ\ar@{-}[r]&\circ\ar@{-}[r]&\circ\ar@{-}[r]&\circ\\
\bullet\ar@{-}[u]&&\circ\ar@{-}[u]&&\bullet\ar@{-}[u]\\
&\bullet\ar@{-}[rr]\ar@{-}[ul]&&\bullet\ar@{-}[ur]&}
$$
This surface is not weakly-minimal. The lines are $x_0=x_3\pm x_1^2=0$ and $x_2=x_3\pm x_1^2=0$.

\item[(\ref{d=2:D5-A1})]
One has $d=2$ and $\Type(X)=\type{D}_5\type{A}_1$.
The~dual graph is
$$
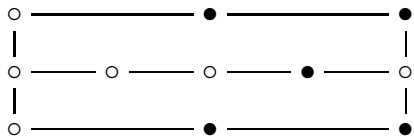
\xymatrix@R=0.8em{
\circ\ar@{-}[rr]\ar@{-}[d]&&\bullet\ar@{-}[rr]&&\bullet\ar@{-}[d]\\
\circ\ar@{-}[r]&\circ\ar@{-}[r]&\circ\ar@{-}[r]&\bullet\ar@{-}[r]&\circ\\
\circ\ar@{-}[rr]\ar@{-}[u]&&\bullet\ar@{-}[rr]&&\bullet\ar@{-}[u]}
$$
The surface $X$ is weakly-minimal. The lines are cut out by $x_0=0$, $x_2=0$ and $x_1=x_3=0$.

\item[(\ref{d=2:2A3})]
One has $d=2$ and $\Type(X)=2\type{A}_3$.
The~dual graph is
$$
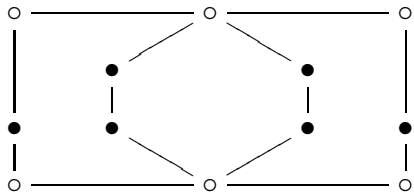
\xymatrix@R=0.8em{
\circ\ar@{-}[rr]&&\circ\ar@{-}[rr]&&\circ\\
&\bullet\ar@{-}[ur]\ar@{-}[d]&&\bullet\ar@{-}[ul]\ar@{-}[d]&\\
\bullet\ar@{-}[d]\ar@{-}[uu]&\bullet\ar@{-}[dr]&&\bullet\ar@{-}[dl]&\bullet\ar@{-}[uu]\ar@{-}[d]\\
\circ\ar@{-}[rr]&&\circ\ar@{-}[rr]&&\circ}
$$
The surface $X$ is weakly-minimal. The lines are cut out by $x_0=0$, $x_1=0$ and $x_2=0$.

\item[(\ref{d=3:E6})]
One has $d=3$ and $\Type(X)=\type{E}_6$.
The~dual graph is
$$
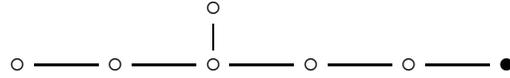
\xymatrix@R=0.8em{
&&\circ\ar@{-}[d]&&&\\
\circ\ar@{-}[r]&\circ\ar@{-}[r]&\circ\ar@{-}[r]&\circ\ar@{-}[r]&\circ\ar@{-}[r]&\bullet}
$$
The surface $X$ is weakly-minimal. The line is $x_0=x_1=0$.

\item[(\ref{d=3:A5-A1})]
One has $d=3$ and $\Type(X)=\type{A}_5\type{A}_1$.
The~dual graph is
$$
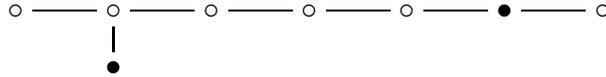
\xymatrix@R=0.8em{
\circ\ar@{-}[r]&\circ\ar@{-}[r]&\circ\ar@{-}[r]&\circ\ar@{-}[r]&\circ\ar@{-}[r]&\bullet&\circ\ar@{-}[l]\\
&\bullet\ar@{-}[u]&&&&&}
$$
The surface $X$ is weakly-minimal. The lines are $x_0=x_1=0$ and $x_1=x_2=0$.

\item[(\ref{d=3:3A2})]
One has $d=3$ and $\Type(X)=3\type{A}_2$.
The~dual graph is
$$
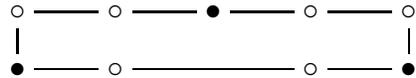
\xymatrix@R=0.8em{
\circ\ar@{-}[d]\ar@{-}[r]&\circ\ar@{-}[r]&\bullet\ar@{-}[r]&\circ\ar@{-}[r]&\circ\\
\bullet\ar@{-}[r]&\circ\ar@{-}[rr]&&\circ\ar@{-}[r]&\bullet\ar@{-}[u]}
$$
The surface $X$ is weakly-minimal. The lines are $x_0=x_3=0$, $x_1=x_3=0$ and $x_2=x_3=0$.

\item[(\ref{d=3:D5})]
One has $d=3$ and $\Type(X)=\type{D}_5$.
The~dual graph is
$$
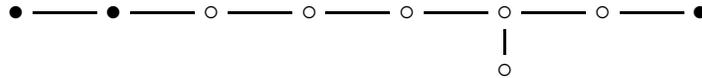
\xymatrix@R=0.8em{
\bullet\ar@{-}[r]&\bullet\ar@{-}[r]&\circ\ar@{-}[r]&\circ\ar@{-}[r]&\circ\ar@{-}[r]&\circ\ar@{-}[r]&\circ\ar@{-}[r]&\bullet\\
&&&&&\circ\ar@{-}[u]&&}
$$
The surface $X$ is weakly-minimal. The lines are $x_0=x_1=0$, $x_0=x_2=0$ and $x_2=x_3=0$.

\item[(\ref{d=3:A5})]
One has $d=3$ and $\Type(X)=\type{A}_5$.
The~dual graph is
$$
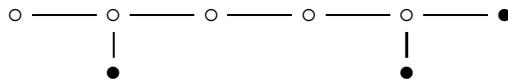
\xymatrix@R=0.8em{
\circ\ar@{-}[r]&\circ\ar@{-}[r]&\circ\ar@{-}[r]&\circ\ar@{-}[r]&\circ\ar@{-}[r]&\bullet\\
&\bullet\ar@{-}[u]&&&\bullet\ar@{-}[u]&}
$$
The surface $X$ is weakly-minimal. The lines are cut out by $x_1=0$.

\item[(\ref{d=3:A4-A1})]
One has $d=3$ and $\Type(X)=\type{A}_4\type{A}_1$.
The~dual graph is
$$
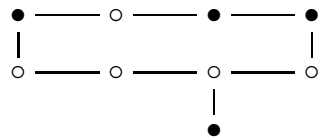
\xymatrix@R=0.8em{
\bullet\ar@{-}[r]\ar@{-}[d]&\circ\ar@{-}[r]&\bullet\ar@{-}[r]&\bullet\ar@{-}[d]
\\
\circ\ar@{-}[r]&\circ\ar@{-}[r]&\circ\ar@{-}[r]&\circ
\\
&&\bullet\ar@{-}[u]&}
$$
The surface $X$ is weakly-minimal. The lines are cut out by $x_0=0$ and $x_1=0$.

\item[(\ref{d=3:A3-2A1})]
One has $d=3$ and $\Type(X)=\type{A}_32\type{A}_1$.
The~dual graph is
$$
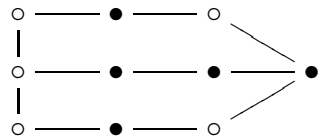
\xymatrix@R=0.8em{
\circ\ar@{-}[r]\ar@{-}[d]&\bullet\ar@{-}[r]&\circ\ar@{-}[rd]&
\\
\circ\ar@{-}[r]&\bullet\ar@{-}[r]&\bullet\ar@{-}[r]&\bullet
\\
\circ\ar@{-}[r]\ar@{-}[u]&\bullet\ar@{-}[r]&\circ\ar@{-}[ur]&}
$$
Then $X$ is not weakly-minimal. The lines are cut out by $x_0=0$, $x_1=0$ and $x_2=0$.

\item[(\ref{d=3:2A2-A1})]
One has $d=3$ and $\Type(X)=2\type{A}_2\type{A}_1$.
The~dual graph is
$$
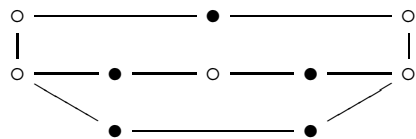
\xymatrix@R=0.8em{
\circ&&\bullet\ar@{-}[ll]\ar@{-}[rr]&&\circ\\
\circ\ar@{-}[u]&\bullet\ar@{-}[l]&\circ\ar@{-}[r]\ar@{-}[l]&\bullet\ar@{-}[r]&\circ\ar@{-}[u]\\
&\bullet\ar@{-}[ul]\ar@{-}[rr]&&\bullet\ar@{-}[ur]&}
$$
The surface $X$ is weakly-minimal. The lines are cut out by $x_1=0$, $x_2=0$ and $x_3=0$.

\item[(\ref{d=3:D4})]
One has $d=3$ and $\Type(X)=\type{D}_4$.
The~dual graph is
$$
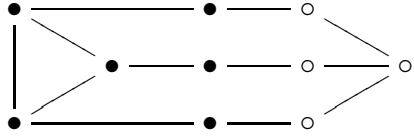
\xymatrix@R=0.8em{
\bullet\ar@{-}[rr]\ar@{-}[dr]&&\bullet\ar@{-}[r]&\circ\ar@{-}[dr]&\\
&\bullet\ar@{-}[r]&\bullet\ar@{-}[r]&\circ\ar@{-}[r]&\circ\\
\bullet\ar@{-}[rr]\ar@{-}[ur]\ar@{-}[uu]&&\bullet\ar@{-}[r]&\circ\ar@{-}[ur]&}
$$
Then $X$ is not weakly-minimal. The lines on $X$ are cut by $x_0=0$ and $x_3=0$.

\item[(\ref{d=3:2A2})]
One has $d=3$ and $\Type(X)=2\type{A}_2$.
The~dual graph is
$$
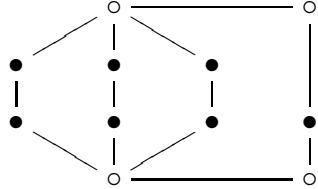
\xymatrix@R=0.8em{
&\circ\ar@{-}[rr]&&\circ\\
\bullet\ar@{-}[ur]\ar@{-}[d]&\bullet\ar@{-}[u]&\bullet\ar@{-}[ul]\ar@{-}[d]&\\
\bullet\ar@{-}[dr]&\bullet\ar@{-}[u]&\bullet\ar@{-}[dl]&\bullet\ar@{-}[uu]\ar@{-}[d]\\
&\circ\ar@{-}[rr]\ar@{-}[u]&&\circ}
$$
The surface $X$ is weakly-minimal. The lines are cut out by $x_1=0$, $x_2=0$ and $x_3=0$.

\item[(\ref{d=4:D5})]
One has $d=4$ and $\Type(X)=\type{D}_5$.
The~dual graph is
$$
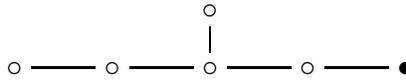
\xymatrix@R=0.8em{
&&\circ\ar@{-}[d]&&\\
\circ\ar@{-}[r]&\circ\ar@{-}[r]&\circ\ar@{-}[r]&\circ\ar@{-}[r]&\bullet}
$$
The surface $X$ is weakly-minimal. The line is $x_2^2-x_1^3=x_0=0$.

\item[(\ref{d=4:A3-2A1})]
One has $d=4$ and $\Type(X)=\type{A}_32\type{A}_1$.
The~dual graph is
$$
\xymatrix{
\circ\ar@{-}[r]&\bullet\ar@{-}[r]&\circ\ar@{-}[r]&\circ\ar@{-}[r]&\circ\ar@{-}[r]&\bullet\ar@{-}[r]&\circ}
$$
The surface $X$ is weakly-minimal. The lines are \mbox{$x_1=x_2=0$} and \mbox{$x_2^2-x_1^3=x_0=0$}.

\item[(\ref{d=4:D4})]
One has $d=4$ and $\Type(X)=\type{D}_4$.
The~dual graph is
$$
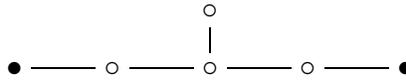
\xymatrix@R=0.8em{
&&\circ\ar@{-}[d]&&\\
\bullet\ar@{-}[r]&\circ\ar@{-}[r]&\circ\ar@{-}[r]&\circ\ar@{-}[r]&\bullet}
$$
The surface $X$ is weakly-minimal, and the~lines are $x_2\pm x_1^2=x_0=0$.

\item[(\ref{d=4:A4})]
One has $d=4$ and $\Type(X)=\type{A}_4$.
The~dual graph is
$$
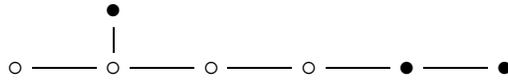
\xymatrix@R=0.8em{
&\bullet\ar@{-}[d]&&&\\
\circ\ar@{-}[r]&\circ\ar@{-}[r]&\circ\ar@{-}[r]&\circ\ar@{-}[r]&\bullet\ar@{-}[r]&\bullet}
$$
Then $X$ is not weakly-minimal. The lines are cut out by $x_0=0$ and $x_1=0$.

\item[(\ref{d=4:A3-A1})]
One has $d=4$ and $\Type(X)=\type{A}_3\type{A}_1$.
The~dual graph is
$$
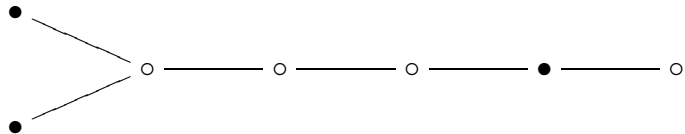
\xymatrix@R=0.8em@C=37pt{
\bullet\ar@{-}[dr]&&&&\\
&\circ\ar@{-}[r]&\circ\ar@{-}[r]&\circ\ar@{-}[r]&\bullet\ar@{-}[r]&\circ\\
\bullet\ar@{-}[ru]&&&&}
$$
The surface $X$ is weakly-minimal. The lines are \mbox{$x_2^2-x_1x_3=x_0=0$} and \mbox{$x_1=x_2\pm x_0^3=0$}.

\item[(\ref{d=4:A2-2A1})]
One has $d=4$ and $\Type(X)=\type{A}_22\type{A}_1$.
The~dual graph is
$$
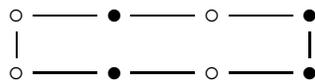
\xymatrix@R=0.8em{
\circ\ar@{-}[r]\ar@{-}[d]&\bullet\ar@{-}[r]&\circ\ar@{-}[r]&\bullet\\
\circ\ar@{-}[r]&\bullet\ar@{-}[r]&\circ\ar@{-}[r]&\bullet\ar@{-}[u]}
$$
The surface $X$ is weakly-minimal. The lines are cut out by $x_2=0$ and $x_3=0$.

\item[(\ref{d=4:4A1})]
One has $d=4$ and $\Type(X)=4\type{A}_1$.
The~dual graph is
$$
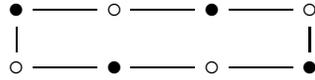
\xymatrix@R=0.8em{
\bullet\ar@{-}[r]\ar@{-}[d]&\circ\ar@{-}[r]&\bullet\ar@{-}[r]&\circ\\
\circ\ar@{-}[r]&\bullet\ar@{-}[r]&\circ\ar@{-}[r]&\bullet\ar@{-}[u]}
$$
The surface $X$ is weakly-minimal. The lines in $X$ are cut by $x_4=0$.

\item[(\ref{d=4:A3-4lines})]
One has $d=4$ and $\Type(X)=\type{A}_3$ and $\numl(X)=4$.
The~dual graph is
$$
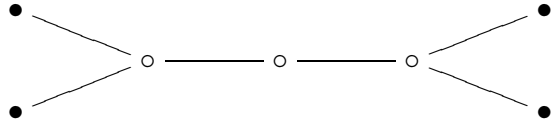
\xymatrix@R=0.6em@C=37pt{
\bullet\ar@{-}[dr]&&&&\bullet\ar@{-}[dl]\\
&\circ\ar@{-}[r]&\circ\ar@{-}[r]&\circ\ar@{-}[rd]&\\
\bullet\ar@{-}[ru]&&&&\bullet}
$$
The surface $X$ is weakly minimal. The lines are $x_2\pm x_1^2=x_0=0$ and $x_2\pm x_0^2=x_1=0$.

\item[(\ref{d=4:A3-5lines})]
One has $d=4$ and $\Type(X)=\type{A}_3$ and $\numl(X)=5$.
The~dual graph is
$$
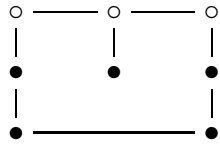
\xymatrix@R=0.9em{
\circ\ar@{-}[r]\ar@{-}[d]&\circ\ar@{-}[r]&\circ\ar@{-}[d]\\
\bullet\ar@{-}[d]&\bullet\ar@{-}[u]&\bullet\\
\bullet\ar@{-}[rr]&&\bullet\ar@{-}[u]}
$$
Then $X$ is not weakly-minimal. The lines are cut out by $x_0=0$ and $x_1=0$.

\item[(\ref{d=4:A2-A1})]
One has $d=4$ and $\Type(X)=\type{A}_2\type{A}_1$.
The~dual graph is
$$
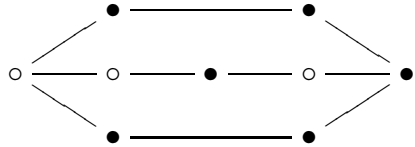
\xymatrix@R=1em{
&\bullet\ar@{-}[r]\ar@{-}[dl]\ar@{-}[rr]&&\bullet\ar@{-}[dr]&\\
\circ\ar@{-}[dr]\ar@{-}[r]&\circ\ar@{-}[r]&\bullet\ar@{-}[r]&\circ\ar@{-}[r]&\bullet&\\
&\bullet\ar@{-}[rr]&&\bullet\ar@{-}[ru]&}
$$
Then $X$ is not weakly-minimal. The lines are cut out by $x_1=0$ and $x_0=0$.

\item[(\ref{d=4:3A1})]
One has $d=4$ and $\Type(X)=3\type{A}_1$.
The~dual graph is
$$
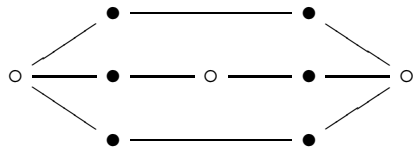
\xymatrix@R=1em{
&\bullet\ar@{-}[r]\ar@{-}[dl]\ar@{-}[rr]&&\bullet\ar@{-}[dr]&\\
\circ\ar@{-}[dr]\ar@{-}[r]&\bullet\ar@{-}[r]&\circ\ar@{-}[r]&\bullet\ar@{-}[r]&\circ&\\
&\bullet\ar@{-}[rr]&&\bullet\ar@{-}[ru]&}
$$
The surface $X$ is weakly-minimal. The lines are cut out by $x_2=0$ and $x_3=0$.

\item[(\ref{d=4:2A1-8lines})]
One has $d=4$ and $\Type(X)=2\type{A}_1$ and $\numl(X)=8$.
The~dual graph is
$$
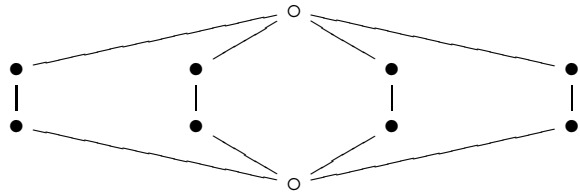
\xymatrix@R=0.8em{
&&&\circ\ar@{-}[llld]\ar@{-}[ld]\ar@{-}[rd]\ar@{-}[rrrd]&&&\\
\bullet\ar@{-}[d]&&\bullet\ar@{-}[d]&&\bullet\ar@{-}[d]&&\bullet\ar@{-}[d]\\
\bullet&&\bullet&&\bullet&&\bullet\\
&&&\circ\ar@{-}[lllu]\ar@{-}[lu]\ar@{-}[ru]\ar@{-}[rrru]&&&}
$$
The surface $X$ is weakly-minimal. The lines are cut out by $x_2=0$ and $x_3=0$.

\item[(\ref{d=5:A4})]
One has $d=5$ and $\Type(X)=\type{A}_4$.
The~dual graph is
$$
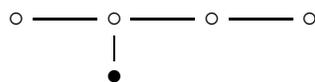
\xymatrix@R=0.8em{
\circ\ar@{-}[r]&\circ\ar@{-}[r]\ar@{-}[d]&\circ\ar@{-}[r]&\circ\\
&\bullet}
$$
The surface $X$ is weakly minimal. The line is given by $x_2^2+x_1^3=x_0=0$.

\item[(\ref{d=5:A3})]
One has $d=5$ and $\Type(X)=\type{A}_3$.
The~dual graph is
$$
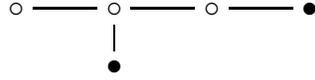
\xymatrix@R=0.8em{
\circ\ar@{-}[r]&\circ\ar@{-}[r]\ar@{-}[d]&\circ\ar@{-}[r]&\bullet\\
&\bullet}
$$
Then $X$ is weakly minimal. The lines are $x_0=x_2=0$ and $y_1=x_2=0$.

\item[(\ref{d=5:A2-A1})]
One has $d=5$ and $\Type(X)=\type{A}_2\type{A}_1$.
The~dual graph is
$$
\xymatrix{
\bullet\ar@{-}[r]&\bullet\ar@{-}[r]&\circ\ar@{-}[r]&\circ\ar@{-}[r]&\bullet\ar@{-}[r]&\circ}
$$
Then surface $X$ is not weakly minimal. It is given in $\PP^5$ by
$$
\left\{\aligned
&x_0x_2=x_1x_5,\\
&x_0x_3=x_5^2,\\
&x_1x_3=x_2x_5,\\
&x_1x_4=x_5^2,\\
&x_2x_4=x_3x_5.\\
\endaligned
\right.
$$
The lines on $X$ are cut out by $x_5=0$.

\item[(\ref{d=5:A2})]
One has $d=5$ and $\Type(X)=\type{A}_2$.
The~dual graph is
$$
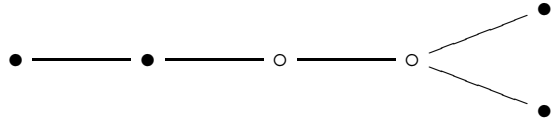
\xymatrix@R=7pt@C=37pt{
&&&&\bullet\\
\bullet\ar@{-}[r]&\bullet\ar@{-}[r]&\circ\ar@{-}[r]&\circ\ar@{-}[ru]\ar@{-}[rd]\\
&&&&\bullet}
$$
Then $X$ is not weakly minimal. The lines are cut out by $x_0=0$ and $x_1=0$.

\item[(\ref{d=5:2A1})]
One has $d=5$ and $\Type(X)=2\type{A}_1$.
The~dual graph is
$$
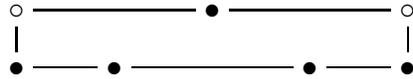
\xymatrix@R=0.8em{
\circ\ar@{-}[d]\ar@{-}[rr]&&\bullet\ar@{-}[rr]&&\circ\ar@{-}[d]\\
\bullet\ar@{-}[r]&\bullet\ar@{-}[rr]&&\bullet\ar@{-}[r]&\bullet}
$$
Then $X$ is not weakly minimal. The lines are cut out $x_0=0$, $x_1=0$ and $x_2=0$.

\item[(\ref{d=5:A1})]
One has $d=5$ and $\Type(X)=\type{A}_1$.
The~dual graph is
$$
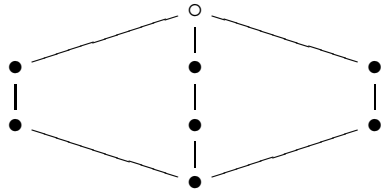
\xymatrix@R=0.8em{
&&\circ\ar@{-}[lld]\ar@{-}[d]\ar@{-}[rrd]&&\\
\bullet\ar@{-}[d]&&\bullet\ar@{-}[d]&&\bullet\ar@{-}[d]\\
\bullet&&\bullet&&\bullet\\
&&\bullet\ar@{-}[llu]\ar@{-}[u]\ar@{-}[rru]&&}
$$
Then $X$ is not weakly minimal. The lines are cut out by $x_0=0$, $x_1=0$, $y_0=0$, $y_1=0$.

\item[(\ref{d=6:A2-A1})]
One has $d=6$ and $\Type(X)=\type{A}_2\type{A}_1$.
The~dual graph is
$$
\xymatrix{
\circ\ar@{-}[r]&\circ\ar@{-}[r]&\bullet\ar@{-}[r]&\circ}
$$
Then $X$ is weakly minimal. The line is $x_0=0$.

\item[(\ref{d=6:A2})]
One has $d=6$ and $\Type(X)=\type{A}_2$.
The~dual graph is
$$
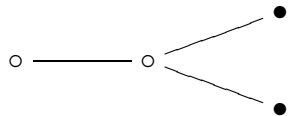
\xymatrix@R=6pt@C=37pt{
&&\bullet\\
\circ\ar@{-}[r]&\circ\ar@{-}[ru]\ar@{-}[rd]\\
&&\bullet}
$$
The surface $X$ is weakly minimal. The lines are $x_0=x_1=0$ and $x_0=x_2=0$.

\item[(\ref{d=6:2A1})]
One has $d=6$ and $\Type(X)=2\type{A}_1$.
The~dual graph is
$$
\xymatrix{
\bullet\ar@{-}[r]&\circ\ar@{-}[r]&\bullet\ar@{-}[r]&\circ}
$$
The surface $X$ is weakly minimal. The lines are cut out by $x_2=0$.

\item[(\ref{d=6:A1-3l})]
One has $d=6$ and $\Type(X)=\type{A}_1$.
The~dual graph is
$$
\xymatrix@R=0.8em{
&\bullet\ar@{-}[d]&\\
\bullet\ar@{-}[r]&\circ\ar@{-}[r]&\bullet}
$$
The surface $X$ is weakly minimal. The lines are cut out by $x_2=0$.

\item[(\ref{d=6:A1-4l})]
One has $d=6$ and $\Type(X)=\type{A}_1$ and $\numl(X)=4$.
The~dual graph is
$$
\xymatrix{
\bullet\ar@{-}[r]&\bullet\ar@{-}[r]&\circ\ar@{-}[r]&\bullet\ar@{-}[r]&\bullet}
$$
The surface $X$ is not weakly minimal. The lines are cut out by $x_0=0$.

\item[(\ref{d=6:smooth})]
One has $d=6$ and $X$ is smooth.
The~dual graph is
$$
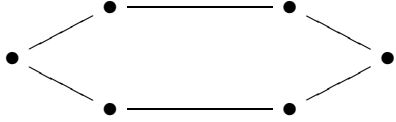
\xymatrix@R=0.6em{
&\bullet\ar@{-}[rr]&&\bullet\ar@{-}[dr]&\\
\bullet\ar@{-}[ur]\ar@{-}[dr]&&&&\bullet\\
&\bullet\ar@{-}[rr]&&\bullet\ar@{-}[ur]&}
$$
Then $X$ is not weakly minimal. The lines are cut out by $x_0=0$, $y_0=0$ and $z_0=0$.
\end{itemize}

% \appendix

\section{Del Pezzo surfaces of homology type of $\PP^2$}
\label{section:appendixb}

In this section, we recall classification of Du Val del Pezzo surfaces whose Weil divisor class group is cyclic \cite{Miyanishi1988,Furushima1986,Ye2002}.
By Lemma~\ref{lemma:Cl}\ref{lemma:Cl:tors-l}, each such surface except $\mathbb{P}^2$ and $\mathbb{P}(1,1,2)$
contains a unique line that generated its class group.

\begin{proposition}
\label{Proposition:Cl=Z}
Let $X$ be a Du Val del Pezzo surface such that $\Cl(X)\cong \ZZ$,
and let~$d:=K_X^2$. Then~$d\neq 7$.
If $d\leqslant 6$, then there is an embedding $X\hookrightarrow\PP(1,2,3,d)$ such that
$X$ is given by
$$
\phi(y_1,y_2,y_3, y_d)=0,
$$
where $\phi$ is a homogeneous polynomial of weighted degree $6$.
If $2\leqslant d\leqslant 6$, then
$$
\phi=y_3^2+y_2^3+y_1^{6-d}y_d,
$$
so that $X$ is uniquely determined by its degree. If $d=1$, there are exactly two possibilities:
\begin{eqnarray}
\label{eq:cor:Cl=Z-2a}
\phi&=& y_3^2+y_2^3+ y_1^5 y_d,\\
\label{eq:cor:Cl=Z-2b}
\phi&=& y_3^2+y_2^3+ y_1^5 y_d+ y_1^2y_2^2,
\end{eqnarray}
which give us two non isomorphic surfaces. The only line $L\subset X$ is cut out by $y_1=0$.
\end{proposition}

\begin{proof}
The proof is similar to the~proof of Theorem~\ref{theorem:index>1}.
\end{proof}

Thus, if $d=9$, $8$, $6$, then $X\cong \PP^2$, $\PP(1,1,2)$, $\PP(1,2,3)$, respectively.

\begin{remark}
\label{remark:Cl1:A2}
In the~notation and assumptions of Proposition~\ref{Proposition:Cl=Z},
one can show that $X\setminus L\cong\Aff^2$.
Moreover, there is cell decomposition $X=\Aff^2\cup \Aff^1\cup \Aff^0$.
In particular, we have
\[
H_q(X,\ZZ)\cong H_q(\PP^2,\ZZ)
\]
for all $q$ (cf. \cite{Bindschadler-Brenton}). In the~case $d=1$, we can say even more: $H^*(X,\ZZ)\cong H^*(\PP^2,\ZZ)$ as rings.
\end{remark}

\begin{remark}
\label{remark:d-1-E8}
Suppose that $X$ is a Du Val del Pezzo surface of degree $1$ such that $\Cl(X)\cong \ZZ$.
By~Proposition~\ref{Proposition:Cl=Z}, $X$ is a~hypersurface in $\PP(1,1,2,3)$ that is given
\begin{enumerate}
\item either by $y_3^2+y_2^3+ y_1^5 y_1^\prime=0$,
\item or by $y_3^2+y_2^3+ y_1^5y_1^\prime+y_1^2y_2^2=0$.
\end{enumerate}
These possibilities are distinguished by the~collection of singular curves in the~pencil $|-K_X|$.
Indeed, in the~first case, the~pencil $|-K_X|$ contains  two singular curves $y_1=0$ and $y_1^\prime=0$, which are both cuspidal.
In the~second case, it has three singular curves $y_1=0$, $y_1^\prime=0$~and~$4y_1+27y_1^\prime=0$.
One of them is cuspidal, and the~remaining two curves are nodal.
\end{remark}

\begin{corollary}
\label{corollary:Cl=Z}
Let $X$ be a Du Val del Pezzo surface, let $d:=K_X^2$.
If $\Cl(X)\cong \ZZ$, then~$\Type(X)$~is
\begin{equation}
\label{eq:sing-E}
\type{E}_8,\quad \type{E}_7,\quad \type{E}_6,\quad \type{D}_5,\quad \type{A}_4,\quad \type{A}_2\type{A}_1,\quad \type{A}_1,\quad \varnothing
\end{equation}
in the~case when $d=1,\, 2,\, 3,\, 4,\, 5,\, 6,\, 8,\, 9$, respectively.
Vice versa, if $\uprho(X)=1$ and $\Type(X)$ is one of the~types \eqref{eq:sing-E}, then $\Cl(X)\cong \ZZ$.
\end{corollary}

\begin{proof}
The first assertion easily follows from Proposition~\ref{Proposition:Cl=Z}.
To prove the~second one, we may assume that $d\leqslant 5$, since the~remaining cases are easy.
Let $\Cl(X,P)$ be the~local Weil divisor class group of the~(unique) singular point $P\in X$. Then in the~cases ~\eqref{eq:sing-E} we have
$$
\Cl(X,P)\cong \ZZ/d\ZZ,
$$
see e.g. \cite[Satz~2.11]{Brieskorn}. Therefore, for any line $L\subset X$ the~divisor $dL$ is Cartier.
Since $dL\simQ -K_X$, we~have $dL\sim -K_X$.
Now, the~assertion follows from Lemma~\ref{lemma:Cl}\ref{lemma:Cl:tors-l} and Proposition~\ref{proposition:ind}.
\end{proof}

For every $d\in\{2,\ldots,6\}$, let $X_d$ be the~Du Val del Pezzo surface of degree $d$ such that $\Cl(X_d)\cong\ZZ$.
As we already mentioned earlier, the~surface $X_d$ contains exactly one line $L_d$ and $\Cl(X_d)=\ZZ[L_d]$.
Take a point $P_d\in L_d\setminus \Sing(X_d)$. Let $\sigma_{d}\colon\widetilde{X}_{d}\to X_d$ be the~blow up of the~point $P_d$,
and let $\widetilde{L}_d$ be the~proper transform on $\widetilde{X}_d$ of the~line $L_d$.
Then
$$
\widetilde{L}_d^2=-1+\frac{1}{d}<0,
$$
and $K_{\widetilde{X}_d}\cdot\widetilde{L}_d=0$.
Therefore, there exists a crepant contraction $\varphi_d\colon\widetilde{X}_d\to X^\prime_{d-1}$ of the~curve $\widetilde{L}_d$,
where $X^\prime_{d-1}$ is a singular Du Val del Pezzo surface such that $\Cl(X^\prime_{d-1})\cong\ZZ$ and $K_{X^\prime_d}^2=K_{\widetilde{X}}^2=d-1$.
Thus, if $d\ne 2$, then $X^\prime_{d-1}\cong X_{d-1}$ by Proposition~\ref{Proposition:Cl=Z}.
If $d=2$, then $X^\prime_1$ is one of the~two surfaces described in Remark~\ref{remark:d-1-E8},
so that we also let $X_1=X^\prime_1$.
Hence, we obtain the~following diagram:
\[
\xymatrix@R=1em{
&
\widetilde  X_{6}\ar[dr]^{\varphi_{6}}\ar[dl]_{\sigma_{6}}&&
\widetilde  X_5\ar[dr]^{\varphi_5}\ar[dl]_{\sigma_5}&&
\widetilde  X_{4}\ar[dr]^{\varphi_{4}}\ar[dl]_{\sigma_{4}}&&
\widetilde  X_3\ar[dr]^{\varphi_3}\ar[dl]_{\sigma_3}&&
\widetilde  X_2\ar[dr]^{\varphi_2}\ar[dl]_{\sigma_2}&&
\\
X_{6}\ar@{-->}[rr]&&X_5\ar@{-->}[rr]&&X_{4}\ar@{-->}[rr]&&X_3\ar@{-->}[rr]&&X_2\ar@{-->}[rr]&&X_1&}
\]
It allows us to reconstruct all surfaces in Proposition~\ref{Proposition:Cl=Z} starting from $X_6=\PP(1,2,3)$.

Each birational transformation $X_{d-1}\dashrightarrow X_{d}$ is $\Aut^0(X_{d-1})$-equivariant,
so that it gives a natural embedding $\Aut^0(X_{d-1})\hookrightarrow\Aut^0(X_{d})$
such that $\Aut^0(X_{d-1})$ is just the~stabilizer of the~point $P_{d}$.
Moreover, there following two assertions hold:
\begin{itemize}
\item if $d\geqslant 3$, then $\Aut^0(X_d)$ transitively acts on $L_d\setminus \Sing(X_d)$;
\item if $d=2$, then $\Aut^0(X_2)\cong \Gm$ has two orbits in $L_2\setminus \Sing(X_2)$,
an open orbit and a closed orbit that consists of a single point, which explains two possibilities in Remark~\ref{remark:d-1-E8}.
\end{itemize}
The construction allow also to compute $\Aut^0(X)$ in the~cases \ref{d=1:E8}, \ref{d=2:E7}, \ref{d=3:E6}, \ref{d=4:D5}, \ref{d=5:A4} of Big Table.

\begin{corollary}
\label{Cor:Cl=Z:Aut}
Let $X$ be a Du Val del Pezzo surface such that $\Cl(X)\cong \ZZ$, and let~$d:=K_X^2$.
Then~the~group $\Aut^0(X)$ is isomorphic to
\[
\BB_3,\quad \BB_2\times \Gm,\quad  \Ga^2\rtimes \Gm,\quad  \Ga\rtimes_{(3)} \Gm,\quad \Gm
\]
in the~case when $d=6$, $5$, $4$, $3$, $2$, respectively.
If $d=1$, then $\Aut^0(X)\cong  \Gm$ if $X$ is given~by~\eqref{eq:cor:Cl=Z-2a},
and $\Aut^0(X)= \{1\}$ if $X$ is given by \eqref{eq:cor:Cl=Z-2b}.
\end{corollary}

\end{document}